\documentclass[10pt]{amsart}

\usepackage{graphicx}

\usepackage[USenglish]{babel}
\usepackage{latexsym}
\usepackage{amsmath}
\usepackage{amsfonts}
\usepackage{amssymb}
\usepackage{esint}
\usepackage[all]{xy}
\renewcommand{\a }{\alpha }
\renewcommand{\b }{\beta }
\renewcommand{\d}{\delta }
\newcommand{\D }{\Delta }

\newcommand{\M}{\mathcal{M}}
\newcommand{\C}{\mathcal{C}}
\newcommand{\e }{\varepsilon }
\newcommand{\g }{\gamma}

\newcommand{\G }{\Gamma }
\renewcommand{\l }{\lambda }
\renewcommand{\L }{\Lambda }

\newcommand{\n }{\nabla }
\newcommand{\var }{\varphi }
\newcommand{\rh }{\rho }

\newcommand{\s }{\sigma }
\newcommand{\Sg }{\Sigma}
\newcommand{\Sig }{\Sigma}

\renewcommand{\th }{\theta }

\renewcommand{\O }{\Omega }

\newcommand{\ov}{\overline}

\newcommand{\wtilde }{\widetilde}

\newcommand{\be}{\begin{equation}}
\newcommand{\ee}{\end{equation}}
\newenvironment{pf}{\noindent{\sc Proof}.\enspace}{\rule{2mm}{2mm}\medskip}

\newtheorem{remark}{Remark}[section]
\newcommand{\R}{\mathbb{R}}
\renewcommand{\S}{\mathbb{S}}
\renewcommand{\P}{\mathbb{P}}

\newcommand{\N}{\mathbb{N}}
\newcommand{\no}{\noindent}
\newcommand{\dis}{\displaystyle}

\newcommand{\T}{\mathbb{T}}
\newcommand{\dkr}{{\bf d}}

\newtheorem{theorem}{Theorem}[section]
\newtheorem{proposition}[theorem]{Proposition}
\newtheorem{example}[theorem]{Example}

\newcommand{\bpr}{\begin{proposition}}
\newcommand{\epr}{\end{proposition}}
\newcommand{\bex}{\begin{example}\rm}
\newcommand{\eex}{\end{example}}

\def\bar#1{\mathcal B_{#1}}
\def\sj#1{\hbox{Sym}^{*{#1}}}
\def\sn{{\mathfrak S_n}}

\begin{document}

\newtheorem{lem}{Lemma}[section]
\newtheorem{pro}[lem]{Proposition}
\newtheorem{thm}[lem]{Theorem}
\newtheorem{rem}[lem]{Remark}
\newtheorem{cor}[lem]{Corollary}
\newtheorem{df}[lem]{Definition}

\title[A general existence result for the Toda system on compact surfaces] {A general existence result for the Toda system \\ on compact surfaces}

\author{Luca Battaglia$^{(1)}$, Aleks Jevnikar$^{(1)}$, Andrea Malchiodi$^{(2)}$, David Ruiz$^{(3)}$}

\address{$^{(1)}$ SISSA, via Bonomea 265, 34136 Trieste (Italy).}

\address{$^{(2)}$  Scuola Normale Superiore, Piazza dei Cavalieri 7, 56126 Pisa,   (Italy).}

\address{$^{(3)}$
Departamento de An\'alisis Matem\'atico, University of Granada,
18071 Granada (Spain).}

\thanks{L.B., A.J. and A.M. are supported by the FIRB project {\em
Analysis and Beyond} and by the PRIN {\em Variational Methods and
Nonlinear PDE's}. A. M. and D.R have been supported by the Spanish
Ministry of Science and Innovation under Grant MTM2011-26717. D.
R. has also been supported by J. Andalucia (FQM 116). L.B. and
A.J. acknowledge support from the Mathematics Department at the
University of Warwick.}

\

\email{lbatta@sissa.it, ajevnika@sissa.it, andrea.malchiodi@sns.it, daruiz@ugr.es}



\begin{abstract}

In this paper we consider the following {\em Toda system} of
equations on a compact surface:
$$ \left\{
    \begin{array}{ll}
      - \D u_1 = 2 \rho_1 \left( \frac{h_1 e^{u_1}}{\int_\Sg
    h_1 e^{u_1} dV_g} - 1 \right) - \rho_2 \left( \frac{h_2 e^{u_2}}{\int_\Sg
    h_2 e^{u_2} dV_g} - 1 \right) - 4 \pi \sum_{j=1}^{m} \a_{1,j} (\d_{p_j}-1), \\
     - \D u_2 = 2 \rho_2 \left( \frac{h_2 e^{u_2}}{\int_\Sg
    h_2 e^{u_2} dV_g} - 1 \right) - \rho_1 \left( \frac{h_1 e^{u_1}}{\int_\Sg
    h_1 e^{u_1} dV_g} - 1 \right) - 4 \pi \sum_{j=1}^{m} \a_{2,j} (\d_{p_j}-1), &
    \end{array}
  \right.$$
which is motivated by the study of models in non-abelian
Chern-Simons theory. Here $h_1, h_2$ are smooth positive
functions, $\rho_1, \rho_2$ two positive parameters, $p_i$ points
of the surface and $\alpha_{1,i}, \alpha_{2,j}$ non-negative
numbers. We prove a general  existence result using variational
methods.

The same analysis applies to the following mean field equation
$$
  - \D u = \rho_1 \left( \frac{h e^{u}}{\int_\Sg
      h e^{u} dV_g} - 1 \right) - \rho_2 \left( \frac{h e^{-u}}{\int_\Sg
      h e^{-u} dV_g} - 1 \right),
$$
which arises in fluid dynamics.
\end{abstract}

\maketitle

\vspace{-0.5cm}

\begin{center}
{\em With an appendix by Sadok Kallel (University of Lille 1)}
\end{center}

\section{Introduction}

\noindent The Toda system
\begin{equation}\label{eq:toda}
    -  \D u_i(x) = \sum_{j=1}^{N} a_{ij} e^{u_j(x)}, \qquad x \in \Sg, \ i = 1, \dots,
    N,
\end{equation}
where $\D$ is the Laplace operator and $A =
(a_{ij})_{ij}$ the {\em Cartan matrix} of $SU(N+1)$,
$$
  A = \left(
        \begin{array}{cccccc}
          2 & -1 & 0 & \dots & \dots & 0 \\
          -1 & 2 & -1 & 0 & \dots & 0 \\
          0 & -1 & 2 & -1 & \dots & 0 \\
          \dots & \dots & \dots & \dots & \dots & \dots \\
          0 & \dots & \dots & -1 & 2 & -1 \\
          0 & \dots & \dots & 0 & -1 & 2 \\
        \end{array}
      \right),
$$
plays an important role in geometry and mathematical physics. In
geometry it appears in the description of holomorphic curves in
$\mathbb{C} \P^n$, see \cite{bw}, \cite{cal}, \cite{cw},
\cite{jw}. In mathematical physics, it is a model for non-abelian
Chern-Simons vortices, which might have applications in
high-temperature superconductivity and which appear in a much
wider variety compared to the Yang-Mills framework, see e.g.
\cite{tar}, \cite{tar3} and \cite{yys} for further details and an
up-to-date set of references.

The existence of abelian Chern-Simons vortices has been quite deeply investigated in the literature,
see e.g. \cite{cay}, \cite{ci}, \cite{ntcv99}, \cite{sy}, \cite{tar96}. The study of the non-abelian
case is more recent, and we refer for example to \cite{dunne1995self}, \cite{kaolee}, \cite{lee},
\cite{noltar2000}.

We will be interested in the following problem on a compact
surface $\Sigma$. For the sake of simplicity,
we will assume that $Vol_g(\Sigma) = 1$.
\begin{equation}\label{eq:e-1}
  \left\{
      \begin{array}{ll}
        - \D u_1 = 2 \rho_1 \left( \frac{h_1 e^{u_1}}{\int_\Sg
      h_1 e^{u_1} dV_g} - 1 \right) - \rho_2 \left( \frac{h_2 e^{u_2}}{\int_\Sg
      h_2 e^{u_2} dV_g} - 1 \right) - 4 \pi \sum_{j=1}^{m} \a_{1,j} (\d_{p_j}-1), \\
       - \D u_2 = 2 \rho_2 \left( \frac{h_2 e^{u_2}}{\int_\Sg
      h_2 e^{u_2} dV_g} - 1 \right) - \rho_1 \left( \frac{h_1 e^{u_1}}{\int_\Sg
      h_1 e^{u_1} dV_g} - 1 \right) - 4 \pi \sum_{j=1}^{m} \a_{2,j} (\d_{p_j}-1). &
      \end{array}
  \right. \end{equation}
Here $h_1, h_2$ are smooth positive functions, and
$\alpha_{i,j}\geq 0$. The above system arises specifically from
gauged self-dual Schr\"odinger equations, see e.g. Chapter 6 in
\cite{yys}: the
Dirac deltas represent {\em vortices} of the wave function, namely
points where the latter vanishes.

To describe the history and the main features of the problem, we first desingularize the
equation using a simple change of variables. Consider indeed the fundamental solution $G_p(x)$
of the Laplace equation on $\Sigma$ with pole at $p$, i.e. the unique solution to
\begin{equation}\label{eq:Green}
   - \D G_p(x) = \d_p - \frac{1}{|\Sig|} \quad \hbox{ on }
   \Sig, \qquad \quad  \mbox{ with } \quad \int_{\Sigma} G_p(x) \, dV_g(x) =0.
\end{equation}
By the substitution
\begin{equation}\label{eq:change}
    u_i(x) \mapsto u_i(x) + 4\pi \sum_{j=1}^m \alpha_{i,j} G_{p_j}(x),\qquad h_i(x)\mapsto \wtilde{h}_i(x)= h_i(x) e^{-4\pi \sum_{j=1}^m \alpha_{i,j} G_{p_j}(x)}
\end{equation}
problem \eqref{eq:e-1} transforms into an equation of the type
\begin{equation} \label{eq:todaregul}
\left\{
    \begin{array}{ll}
      - \D u_1 = 2 \rho_1 \left( \frac{\wtilde{h}_1 e^{u_1}}{\int_\Sg
    \wtilde{h}_1 e^{u_1} dV_g} - 1 \right) - \rho_2 \left( \frac{\wtilde{h}_2 e^{u_2}}{\int_\Sg
    \wtilde{h}_2 e^{u_2} dV_g} - 1 \right), \\
     - \D u_2 = 2 \rho_2 \left( \frac{\wtilde{h}_2 e^{u_2}}{\int_\Sg
    \wtilde{h}_2 e^{u_2} dV_g} - 1 \right) - \rho_1 \left( \frac{\wtilde{h}_1 e^{u_1}}{\int_\Sg
    \wtilde{h}_1 e^{u_1} dV_g} - 1 \right), &
    \end{array}
  \right.
\end{equation}
where the functions $\wtilde{h}_j$ satisfy
\begin{equation} \label{eq:h}
\wtilde{h}_i > 0 \quad \hbox{ on } \Sigma \setminus \{p_1, \dots,
p_{m}\}; \qquad \quad
  \wtilde{h}_i(x) \simeq d(x,p_j)^{2 \a_{i,j}},\mbox{ near } p_j,\ i=1,\ 2.
\end{equation}

Problem \eqref{eq:todaregul} is variational, and solutions can be
found as critical points of the Euler-Lagrange functional $J_{\rho} : H^1(\Sg) \times
H^1(\Sg) \to \R$ ($\rho=(\rho_1,\rho_2)$) given by
\begin{equation} \label{funzionale}
J_{\rho}(u_1, u_2) = \int_\Sg Q(u_1,u_2)\, dV_g  + \sum_{i=1}^2
\rho_i \left ( \int_\Sg u_i dV_g - \log \int_\Sg \wtilde{h}_i e^{u_i} dV_g
\right ),
\end{equation}
 where $Q(u_1,u_2)$ is defined as:
\begin{equation}\label{eq:QQ}
Q(u_1,u_2) = \frac{1}{3} \left ( |\n u_1|^2 + |\n u_2|^2 + \n u_1 \cdot
\n u_2\right ).
\end{equation}

A basic tool for studying functionals like $J_{\rho}$ is the
Moser-Trudinger inequality, see \eqref{eq:mt}.  Its analogue for
the Toda system has been obtained in \cite{jw} and reads as
\begin{equation} \label{mtjw}  4\pi \sum_{i=1}^2 \left (\log \int_\Sg h_i e^{u_i} dV_g -
\int_\Sg u_i dV_g \right ) \leq \int_\Sg Q(u_1,u_2)\, dV_g +C \qquad \quad
\forall \, u_1, u_2 \in H^1(\Sigma),
\end{equation}
for some $C=C(\Sg)$. This inequality immediately allows to find a
global minimum of $J_{\rho}$ provided both $\rho_1$ and $\rho_2$
are less than $4 \pi$. For larger values of the parameters
$\rho_i$ $J_{\rho}$ is unbounded from below and the problem
becomes more challenging. In this paper we use min-max theory to
find a critical point of $J_\rho$ in a general non-coercive
regime. Our main result is the following:

\begin{thm} \label{t:ex} Let $\Lambda \subset \R^2$ be as in Definition \ref{d:glob}.
Let $\Sigma$ be a compact surface neither homeomorphic to $\S^2$ nor to $\R\P^2$, and assume that $(\rho_1, \rho_2) \not\in \Lambda$. Then
\eqref{eq:e-1} is solvable.
\end{thm}

Let us point out that $\Lambda \subseteq \R^2$ is an explicit set
formed by an union of straight lines, see
Remark \ref{r:discr}. In particular it is a closed set with zero
Lebesgue measure.

\

Up to our knowledge, \emph{there is no previous existence result
in the literature} for the singular Toda system. Our result is
hence the first one in this direction, and is generic in the
choice of parameters $\rho_1$ and $\rho_2$. In the regular
case there are some previous existence results, see \cite{jlw, cheikh, mr2}, some of which have a counterpart in
\cite{djlw} and \cite{dja} for the scalar case \eqref{scalar} (see also \cite{dm}
for a higher order problem and \cite{bt1, bm, carma, malrui} for
the singular case). However, these require an upper bound either on one of the $\rho_i$'s
or both: hence our result covers most of the unknown cases also for the
regular problem.

The main difficulties in attacking \eqref{eq:todaregul} are mainly of
two kinds: compactness issues and the Morse-structure of the
functional, which we are going to describe below.

\

\noindent As many geometric problems, also  \eqref{eq:todaregul}
presents loss of compactness phenomena, as its solutions might
blow-up. To describe the general  phenomenon it is first
convenient to discuss the case of the scalar counterpart of
\eqref{eq:todaregul}, namely is a Liouville equation the form:
\begin{equation} \label{scalar}
 - \Delta u = 2 \rho \left( \frac{h \,  e^{u}}{\int_\Sigma h \,
e^{u} d V_g} - 1 \right),\end{equation} where $\rho \in \R$ and
where $h(x)$ behaves as in \eqref{eq:h} near the singularities.
Equation \eqref{scalar} rules the change of Gaussian curvature
under conformal deformation of the metric, and it also
describes the abelian counterpart of \eqref{eq:e-1} from the
physical point of view. This equation has been very much studied
in the literature; there are by now many results regarding
existence, compactness of solutions, bubbling behavior, etc. We
refer the interested reader to the reviews \cite{mreview, tar3}.

Concerning \eqref{scalar} it was proved in \cite{breme}, \cite{liyy} and \cite{ls} that for the regular case
a blow-up point $\ov{x}_R$ for a sequence $(u_n)_n$ of solutions satisfies the following quantization
property
\begin{equation} \label{eq:quantscal}
\lim_{r \to 0} \lim_{n \to + \infty} \rho \int_{B_r(\ov{x}_R)} h \, e^{u_n} dV_g = 4 \pi,
\end{equation}
and that the limit profile of solutions is that of a {\em bubble},
namely the logarithm of the conformal factor of the stereographic
projection from $S^2$ onto $\R^2$, composed with a dilation.

For the singular case instead, it was proven in \cite{bt1} and \cite{bmo} that if blow-up occurs at a
singular point $\ov{x}_S$ with weight $- 4 \pi \a$  then one has
\begin{equation} \label{eq:quantscalsing}
\lim_{r \to 0} \lim_{n \to + \infty} \rho \int_{B_r(\ov{x}_S)} h \, e^{u_n} dV_g = 4 \pi (1+\a),
\end{equation}
whereas \eqref{eq:quantscal} still holds true if blow-up occurs at
a regular point.

\

\noindent This  behaviour helps to explain the blow-up feature for
 system \eqref{eq:todaregul}, which inherits some character
from the scalar case. Consider first the regular case, that is,
\eqref{eq:e-1} with $\a_{i,j}=0$. Here a sequence of solutions can
blow-up in three different ways: one component blows-up and the
other does not; one component blows-up faster than the other; both
components blow-up at the same rate.

It was proved in \cite{jlw, jw2} that the quantization values for
the two components are respectively $(4\pi, 0)$ or $(0, 4 \pi)$ in
the first case, $(8 \pi, 4 \pi)$ or $(4 \pi, 8 \pi)$ in the second
case and $(8 \pi, 8 \pi)$ in the third one. Notice that, by the
results in \cite{dekomu}, \cite{esgrpi} and \cite{mupiwe}, all the
five alternatives may indeed happen.

\

When singular sources are present a similar phenomenon happens,
which has been investigated in the recent paper
\cite{linweizhang}. If blow-up occurs at a point $p$ with values
$\a_{1}, \a_2$ (we may allow them to vanish), the corresponding
blow-up values would be
$$
  (4\pi(1+\a_1, 0)); \qquad (0, 4\pi(1+\a_2)); \qquad (4 \pi (1+\a_1), 4 \pi(2 + \a_1 + \a_2));
$$
$$
  (4 \pi(2 + \a_1 + \a_2), 4 \pi (1 + \a_2)); \qquad (4\pi (2 + \a_1 + \a_2), 4\pi (2 + \a_1 + \a_2)).
$$
Other (finitely-many) blow-up values are indeed allowed, see
Theorem \ref{t:lwz} for details, as more involved situations are
not yet excluded (or known to exist). As a consequence, the set of
solutions to \eqref{eq:todaregul} is compact whenever
$(\rho_1,\rho_2) \notin \Lambda$: this is the main reason for our
assumption in  Theorem \ref{t:ex}.

\

Let us now show how we can study the sub-levels of the functional
and conclude existence of solutions via min-max methods.  The main tool in the
variational study of this kind of problems is the so-called
Chen-Li inequality, see \cite{cl}. In the scalar case, it implies
that  a suitable {\em spreading} of the term $e^u$ yields a better constant in the Moser-Trudinger inequality, which
in turn might imply a lower bound on the Euler functional $I_\rho$
of \eqref{scalar}
\begin{equation}\label{scalar2}
    I_\rho(u) =  \frac 1 2 \int_{\Sigma} |\n_g u|^2 dV_g + 2 \rho \left ( \int_\Sigma
    u \, dV_g -  \log \int_\Sigma h \, e^{u} dV_g \right),
    \qquad u \in H^{1}(\Sigma).
\end{equation}
The consequence of this fact is that if $\rho < 4 (k+1) \pi$, $k \in \N$, and if $I_\rho(u)$ is
large negative (i.e. when lower bounds fail) $e^u$ accumulates near at most $k$ points of $\Sigma$,
see e.g. \cite{dja}.
This suggests to introduce the family of unit measures $\Sigma_k$ which are supported in at
most $k$ points of $\Sigma$, known as {\em formal barycenters} of $\Sigma$
\begin{equation}
\label{selru}
\Sigma_{k}=\left\{\sum_{j = 1}^k t_{j}\delta_{x_{j}}:\  \sum_{j=1}^kt_{j}=1, \ x_{j}\in\Sigma \right\}.
\end{equation}
One can show that, for any integer $k$, $\Sigma_k$ is not
contractible and that its homology is mapped injectively into that
of the low sub-levels of $I_\rho$. This allows to prove existence
of solutions via suitable min-max schemes.

\

\noindent When both $\rho_1$ and $\rho_2$ are larger than $4 \pi$
the description of the sub-levels becomes more involved, since the
two components $u_1$ and $u_2$ interact in a non-trivial way. See
\cite{mr2} on this respect. In this paper we obtain a {\em partial} topological
characterization of the low energy levels of $J_\rho$, which is
however sufficient for our purposes. This strategy has been used
in \cite{bdm} and in \cite{bardem} for the singular scalar
equation and for a model in electroweak theory respectively, while
in this paper the general non-abelian case is treated for the
first time.

First, we construct two disjoint simple non-contractible
curves $\gamma_1, \gamma_2$ which do not intersect singular
points, and define global retractions $\Pi_1, \Pi_2$ of $\Sigma$
onto these two curves. Such curves do not exist for $\Sigma= \S^2$ or $\R\P^2$, and hence our arguments do not work in those cases.

Combining arguments from \cite{cl},
\cite{cheikh} and \cite{mr2} we prove that if $\rho_1 < 4 (k+1)
\pi$ and $\rho_2 < 4 (l+1) \pi$, $k, l \in \N$, then either
$\wtilde{h}_1 e^{u_1}$ is close to $\Sigma_k$ or $\wtilde{h}_2
e^{u_2}$ is close to $\Sigma_l$ in the distributional sense. Then
we can map continuously (and naturally) $\wtilde h_1e^{u_1}$ to
$\Sigma_k$ or $\wtilde h_2e^{u_2}$ to $\Sigma_l$; using then the
retractions $\Pi_i$ one can restrict himself to targets in
$(\gamma_1)_k$ or $(\gamma_2)_l$ only. This alternative can be
expressed naturally in terms of the {\em topological join}
$(\gamma_1)_k * (\gamma_2)_l$. Roughly speaking, given two
topological spaces $A$ and $B$, the join $A * B$ is the formal set
of segments joining elements of $A$ with elements of $B$, see
Section \ref{s:prel} for details. In this way, we are able to
define a global projection $\Psi$ from low sub-levels of $J_\rho$
onto $(\gamma_1)_k * (\gamma_2)_l$.

We can also construct a reverse map $\Phi_\lambda$ (where
$\lambda$ is a large parameter) from $(\gamma_1)_k * (\gamma_2)_l$
into arbitrarily low sub-levels of $J_\rho$ using suitable test
functions. Moreover, we show that the composition of both maps is
homotopic to the identity map. Finally, $(\gamma_1)_k
* (\gamma_2)_l$ is homeomorphic to a sphere of dimension $2k +2 l
- 1$ see Remark \ref{r:sphere}: in particular it is not
contractible, and this allows us to apply a min-max argument.

In this step a compactness property is needed, like the
Palais-Smale's. The latter is indeed not known for this problem,
but there is a way around it using a monotonicity method from
\cite{struwe88}. For that, compactness of solutions comes to rescue, and
here we use the results of \cite{jlw} and \cite{linweizhang}. This
is the reason why we assume $(\rho_1, \rho_2) \notin \Lambda$.

\

\noindent In this paper we also give a general result for a mean
field equation, Theorem \ref{t:ex2}, arising from models in fluid
dynamics and in the description of constant mean curvature
surfaces: to keep the introduction short we discuss its motivation
and how our result compares to the existing literature in Section
\ref{s:sh}.

\

\noindent The plan of the paper is the following: in Section
\ref{s:prel} we recall some preliminary results on Moser-Trudinger
inequalities, the notion of topological join and a compactness
theorem. In Section \ref{s:test} we construct a family of test
functions with low energy modelled on the topological join of
$(\gamma_1)_k$ and $(\gamma_2)_l$. In Section \ref{s:mtineq} we
derive suitable improved Moser-Trudinger inequalities to construct
projections from low sub-levels of $J_\rho$ into $(\gamma_1)_k *
(\gamma_2)_l$. In Section \ref{s:minmax} we prove our existence
theorem using the min-max argument and finally in Section
\ref{s:sh} we discuss the mean field equation.

\

\noindent {\bf Acknowledgment:} This paper includes an appendix of
Sadok Kallel which establishes that $\Sigma_k$ is a CW-complex.
This allows us to give a self-consistent and short proof of
Proposition \ref{p:projbar}. The authors are deeply grateful to
him for this contribution.

\

\noindent {\bf Added in Proof:} After this paper has been completed, further existence results have appeared in the literature. In \cite{shadow} the Leray-Schauder degree is computed for the singular Toda System if $\min \{\rho_1,\rho_2\} < 8 \pi$. As a consequence the authors obtain existence results. The computation of the degree is made via a careful description of the blow-up solutions for the Toda System exhibiting partial blow-up (this type of solutions has also been found in \cite{dpr1}). We point out that our existence result in this paper has no limitation on the sizes of $\rho_1$, $\rho_2$.

In \cite{jev-kal-mal} variational methods used to prove existence if $\min \{\rho_1,\rho_2\} < 8 \pi$ but without any restriction on the topology of $\Sigma$. The cases of $\alpha_j$ negative has been considered in \cite{batjmaa, bat-mal}, where existence and non-existence results are given.

\section{Notation and preliminaries} \label{s:prel}

\noindent In this section we collect some useful notation and preliminary material. The Appendix at the end of the paper use independent notation, which will be established there.

Given points $x, y \in \Sg$,  $d(x,y)$ will stand for the metric distance
between $x$ and $y$ on $\Sg$. Similarly, for any $p \in \Sg
$, $\Omega, \Omega' \subseteq \Sg$, we set:

$$ d(p, \Omega) = \inf \left\{ d(p,x) \; :  x \in \Omega
\right\}, \qquad
  d(\Omega,\Omega') = \inf \left\{ d(x,y) \; : \; x \in \Omega,\ y \in
  \Omega'
\right\}.
$$
The symbol $B_s(p)$ stands for the open metric ball of
radius $s$ and centre $p$, and the complement of a set $\Omega$ in
$\Sg$ will be denoted by $\Omega^c$.

Given a function $u \in L^1(\Sg)$ and $\Omega \subset \Sg$,  the average of $u$ on $\Omega$
is denoted by the symbol
$$ \fint_{\Omega} u \, dV_g = \frac{1}{|\Omega|} \int_{\Omega} u \,
dV_g.$$ We denote by $\ov{u}$ the average of $u$ in $\Sg$: since
we are assuming $|\Sg| = 1$, we have
$$
 \ov{u}= \int_\Sg u \, dV_g = \fint_\Sg u\, dV_g.
$$
The sub-levels of the functional $J_\rho$ will be indicated as
$$
J_\rho^{a}:=\left\{u=(u_1,u_2)\in H^1(\Sigma)\times H^1(\Sigma):\;J(u_1,u_2)\le a\right\}
$$
Throughout the paper the letter $C$ will stand for large constants which
are allowed to vary among different formulas or even within the same lines.
When we want to stress the dependence of the constants on some
parameter (or parameters), we add subscripts to $C$, as $C_\d$,
etc. We will write $o_{\alpha}(1)$ to denote
quantities that tend to $0$ as $\alpha \to 0$ or $\alpha \to
+\infty$; we will similarly use the symbol
$O_\a(1)$ for bounded quantities.

\

\noindent We recall next the classical Moser-Trudinger inequality, in its weak form
\begin{equation}\label{eq:mt}
 \log\int_{\Sigma}e^{u-\overline{u}}\,dV_{g}\leq\frac{1}{16 \pi}\int_{\Sigma}\left|\nabla_{g}u\right|^{2}\,dV_{g}+C;
 \qquad \quad u \in H^1(\Sigma),
\end{equation}
where $C$ is a constant depending only on $\Sigma$ and the metric $g$. For the Toda system, a similar
sharp inequality was derived in \cite{jw} concerning the regular case: indeed, since the weights $\a_{ij}$ are
positive, that inequality applies to the singular case as well, as the functions $\wtilde{h}_i$  are
uniformly bounded.


\begin{thm}\label{th:jw} (\cite{jw})
The functional $J_\rho$ is bounded from below if and only if $\rho_i \leq 4 \pi$, $i=1, 2$.
\end{thm}

\

\noindent As it is mentioned in the introduction, some useful
information arising from Moser-Trudinger type inequalities and
their improvements are the concentration of $e^{u_i}$ when
$u=(u_1,u_2)$ belongs to a low sub-level. To express this
rigorously, we denote $\M(\Sigma)$ the set of all Radon measures
on $\Sigma$, and introduce a norm by using duality versus Lipschitz functions, that is, we set:
\begin{equation}
\label{distsup} \| \mu  \|_{Lip'(\Sigma)} =
\sup_{\left\|f\right\|_{Lip\left(\Sigma\right)}\leq 1}
\left| \int_\Sigma f \, d\mu \right|; \qquad \quad \mu, \nu \in \M(\Sigma).
\end{equation}
We denote by $\dkr$ the corresponding distance, which receives the
name of Kantorovich-Rubinstein distance.

When a measure is close in the $Lip'$ sense to an element in
$\Sigma_k$, see \eqref{selru}, it is then possible to map it
continuously to a nearby element in this set. This
has been proved in \cite{dm}, but we give here a much shorter and
self-consistent proof.

\begin{pro}\label{p:projbar}
Given $k\in\N$, for $\e_0 $ sufficiently small there exists a
continuous retraction:
$$\psi_k: \{ \s \in \M(\Sigma),\ \dkr(\s , \Sigma_k) < \e_0 \} \to \Sigma_k.$$

Here continuity is referred to the distance $\dkr$. In particular,
if $\s_n \rightharpoonup \s $ in the sense of measures, with $\s
\in \Sigma_k$, then $\psi_k(\s_n) \to \s$.
\end{pro}

\begin{pf} Observe that the inclusion $Lip(\Sigma) \subset C(\Sigma)$
is compact: therefore, $\M(\Sigma)=C(\Sigma)' \subset
Lip(\Sigma)'$ is also compact. Of course, the set $\Sigma_k
\subset \M(\Sigma)$, and then it is inside $Lip(\Sigma)'$. Since
$\Sigma_k$ is a Euclidean Neighbourhood Retract (ENR) (see
Appendix E of \cite{bredon}), there exists a neighbourhood $V
\supset \Sigma_k$ in the $Lip'$ topology, and a continuous
retraction $\psi_k: V \to \Sigma_k$.

Now, if $\s_n \rightharpoonup \s \in \Sigma_k$ in the sense of
measures, by compactness, $f_n \to \s$ in $Lip'$, and by
continuity, $\psi_k(f_n) \to \psi_k(\s)$. But, since $\psi_k$ is a
retraction, $\psi_k(\s)=\s$.
\end{pf}

\begin{remark} \label{kallel} In the Appendix to this paper Sadok Kallel proves
that $\Sigma_k$ is a CW-complex. As a consequence it is an
Euclidean Neighborhood Retract, see for instance Appendix E of
\cite{bredon}. And this is the key point of the proof of
Proposition \ref{p:projbar}.

\end{remark}

\noindent At some point of our proof we will be under the  assumptions of Proposition \ref{p:projbar}
for either $f = \wtilde{h}_1 e^{u_1}$ or for $f = \wtilde{h}_2 e^{u_2}$. To deal with this alternative it
will then be convenient to use the notion of {\em topological join}, which we recall here.
The topological join of two sets $A, B$ is defined as the family of elements of the form
$$
 \frac{ \left\{ (a,b,r): \; a \in A,\; b \in B,\; r \in [0,1]  \right\}}R,
$$
where $R$ is an equivalence relation such that
$$
(a_1, b,1) \stackrel{R}{\sim} (a_2,b, 1)  \quad \forall a_1, a_2
\in A, b \in B
  \qquad \quad \hbox{and} \qquad \quad
(a, b_1,0)  \stackrel{R}{\sim} (a, b_2,0) \quad \forall a \in A,
b_1, b_2 \in B.
$$

The elements of the join are usually written as formal sums
$(1-r)a+rb$.

\

\noindent The next tool we will need is a compactness result which follows from \cite{linweizhang,batman}: before stating it
it is convenient to introduce a finite set of couples of numbers, which represent possible
quantization values for the concentration of the exponential functions. Consider a  point $p$
at which \eqref{eq:e-1} has singular weights $\a_1 = \a_1(p), \a_2 = \a_2(p)$
in the first and the second component of the
equation. We give then the following two definitions.

\begin{df}\label{d:cvalues}
Given a couple of non-negative numbers $(\alpha_1, \alpha_2)$ we let
$\Gamma_{\a_1, \a_2}$ be the subset of an ellipse in $\R^2$ defined by the equation
$$
  \Gamma_{\a_1, \a_2} := \left\{ (\s_1, \s_2) \; : \; \s_1, \s_2 \geq 0,
  \s_1^2 - \s_1 \s_2 + \s_2^2 = 2 (1 + \a_1) \s_1 + 2 (1 + \a_2) \s_2    \right\}.
$$
We then let  $\Lambda_{\a_1, \a_2} \subseteq \Gamma_{\a_1, \a_2}$ be the set constructed via the
following rules:

\

\noindent {\bf 1.} the points $(0,0)$,
$(2(1+\a_1, 0)), (0, 2(1+\a_2)),  (2 (1+\a_1), 2 (2 + \a_1 + \a_2))$,
$(2 (2 + \a_1 + \a_2), 2 (1 + \a_2))$, $(2 (2 + \a_1 + \a_2), 2 (2 + \a_1 + \a_2))$ belong to
$\Lambda_{\a_1, \a_2}$;

\

\noindent {\bf 2.} if $(a, b) \in \Lambda_{\a_1, \a_2}$ then also any  $(c,d) \in \Gamma_{\a_1, \a_2}$
with $c = a + 2 m$,  $m  \in \N \cup \{ 0 \}$, $d \geq b$  belongs to $\Lambda_{\a_1, \a_2}$;

\

\noindent {\bf 3.} if $(a, b) \in \Lambda_{\a_1, \a_2}$ then also any  $(c,d) \in \Gamma_{\a_1, \a_2}$
with $d = b + 2 n$,  $n \in \N \cup \{ 0 \}$, $c \geq a$  belongs to $\Lambda_{\a_1, \a_2}$.
\end{df}

\

\begin{df} \label{d:glob} Given $\Lambda_{\a_1, \a_2}$ as in Definition  \ref{d:cvalues}, we set
$$
  \Lambda_i= 2 \pi \left\{ 2 n + \sum \limits_{j=1}^m a_j n_j,\;n\in\N\cup\{0\},\
  n_j\in\{0,1\} ,\;a_j\in\pi_i(\G_{\a_{1,j},\a_{2,j}})\right\}, \qquad i = 1, 2.
$$
where $\pi_i$ is the projection on the $i^{th}$ component.
We finally set
$$
   \Lambda = (\Lambda_1 \times \R) \cup (\R \times \Lambda_2) \subseteq \R^2.
$$
\end{df}

\

\noindent In \cite{linweizhang} it was proved that the local quantization values for blow-up at a point $p_j$ must belong to $\Lambda_{\a_{1,j},\a_{2,j}}$. Anyway, this does not suffice to get a global compactness result because, a priori, there could be some non-vanishing residual mass. In \cite{dpr1, dpr2} the authors construct sequences of blowing-up solutions for the Toda System where one component has some residual mass.\\
Such an issue was solved in \cite{batman}; given a blowing-up sequence of solutions of \eqref{eq:e-1}, at least one of the components must have zero residual mass. This easily implies that in case of blow-up one of the $\rho_i$ must be a finite combination of the local blow-up values, that is, it must belong to $\Lambda_i$.

\begin{thm}\label{t:lwz} (\cite{linweizhang,batman}) For $(\rho_1, \rho_2)$ in a fixed compact set
of $\R^2 \setminus \L$ the family of solutions to \eqref{eq:todaregul} is uniformly bounded in
$C^{2,\b}$ for some $\b > 0$.
\end{thm}

\

\begin{rem}\label{r:discr} Observe that $\L_{\a_1, \a_2}$ is finite,
and it coincides with the five elements $(4\pi, 0)$, $(0, 4 \pi)$,
$(8 \pi, 4 \pi)$, $(4 \pi, 8 \pi)$, $(8 \pi, 8 \pi)$ when both
$\a_1$ and $\a_2$ vanish. The quantization for the regular Toda system was proved in \cite{jlw}.
\end{rem}

\section{The topological set and test functions} \label{s:test}

We begin this section with an easy topological result, which will be essential in our analysis:

\begin{lem} \label{new} Let $\Sigma$ be a compact surface not homeomorphic to $\S^2$ nor $\R\P^2$. Then, there exist two
simple closed curves $\gamma_1, \gamma_2 \subseteq \Sigma$
satisfying (see Figure \ref{fig:torogammai})

\begin{enumerate}
\item $\gamma_1, \gamma_2$ do not intersect
each other nor any of the singular points $p_j$, $j=1 \dots m$;

\item there exist global retractions $\Pi_i: \Sigma \to \gamma_i$,
$i=1,2$.

\end{enumerate}

\end{lem}

\begin{proof}

The result is quite evident for the torus. For the Klein bottle, consider its fundamental square $ABAB^{-1}$. there is no previous existence result in the literature for the singular Toda
system.We can take $\gamma_1$ as the segment $B$, and $\gamma_2$ a segment parallel to $B$ and passing by the center of the square. The retractions are given by just freezing one cartesian component of the point in the square.

Observe that we can assume that $p_i$ do not intersect those curves.

For any other $\Sigma$ under the conditions of the lemma, Dyck's Theorem implies that it is the connected sum of a torus and another compact surface, $\Sigma= \T^2 \# M$. Then, one can modify the retractions of the torus so that they are constant on $M$.
\end{proof}

\begin{remark} Observe that each curve $\gamma_i$ generates a free subgroup in the first co-homology group of $\Sigma$. Then, Lemma \ref{new} cannot hold for $\S^2$ or $\R \P^2$.
\end{remark}

\begin{figure}[h]
\centering
\includegraphics[width=0.7\linewidth]{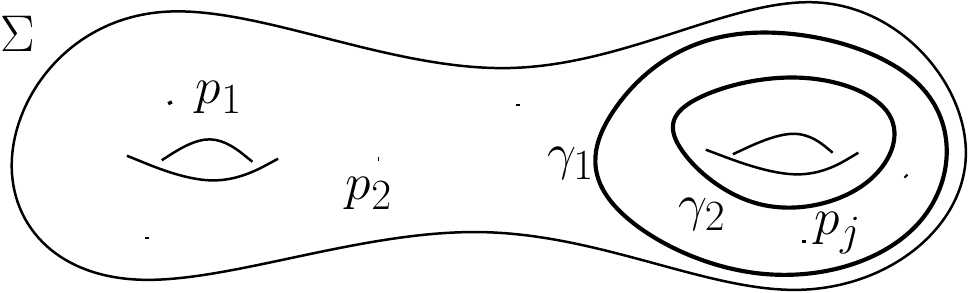}
\caption{The curves $\gamma_i$}
\label{fig:torogammai}
\end{figure}

\noindent For $\rho_1 \in (4 k \pi, 4 (k+1) \pi)$ and $\rho_2 \in
(4 l \pi, 4 (l+1) \pi)$ we would like to build a family of
test functions modelled on the topological join $ (\gamma_1)_k
* (\gamma_2)_l$, involving the formal barycenters of the curves
$\gamma_1, \gamma_2$, see \eqref{selru}.

\begin{rem} \label{r:sphere}
Since each $\gamma_i$ is homeomorphic to $\S^1$, it follows from
Proposition 3.2 in \cite{bm} that $(\gamma_1)_k$ is homeomorphic
to $\S^{2k-1}$ and $(\gamma_2)_l$ to $\S^{2l-1}$ (the homotopy
equivalence was found before in \cite{kk}). As it is well-known,
the join $S^m
* S^n$ is homeomorphic to $S^{m+n+1}$ (see for example
\cite{hat}), and therefore $ (\gamma_1)_k * (\gamma_2)_l$ is
homeomorphic to the sphere $\S^{2k+2l-1}$.
\end{rem}

\

\noindent Let $\zeta=(1-r)\s_2+r\s_1\in (\gamma_1)_k * (\gamma_2)_l$, where:
$$
  \s_1 := \sum_{i=1}^k t_i \d_{x_i} \in (\gamma_1)_k \qquad \quad
  \hbox{ and } \qquad \quad \s_2 :=  \sum_{j=1}^l s_j \d_{y_j} \in
  (\gamma_2)_l.
$$
Our goal is to define a test function modelled on any $\zeta
\in (\gamma_1)_k * (\gamma_2)_l$, depending on a positive parameter $\lambda$ and belonging to low sub-levels of $J$ for large $\lambda$, that
is a map
$$
\Phi_\lambda: (\gamma_1)_k * (\gamma_2)_l \to J_\rho^{-L}; \qquad \quad L \gg 0.
$$
For any $\l > 0$, we define the parameters
$$
  \l_{1,r} = (1-r) \l; \qquad \qquad \l_{2,r} = r \l.
$$
We introduce $\Phi_\lambda(\zeta)= \var_{\l , \zeta}$
whose components are defined by
\begin{equation}
\label{e:pl}
  \left( \begin{array}{c} \var_1(x)
   \\ \var_2(x)
  \end{array} \right)
   =  \left(  \begin{matrix}
    \log \, \sum_{i=1}^{k} t_i \left( \frac{1}{1 + \l_{1,r}^2 d(x,x_i)^2} \right)^2
    - \frac{1}{2} \log  \, \sum_{j=1}^{l} s_j \left( \frac{1}{1 + \l_{2,r}^2 d(x,y_j)^2} \right)^2
   \\ - \frac{1}{2} \log \, \sum_{i=1}^{k} t_i \left( \frac{1}{1 + \l_{1,r}^2 d(x,x_i)^2} \right)^2
      + \log \, \sum_{j=1}^{l} s_j \left( \frac{1}{1 + \l_{2,r}^2 d(x,y_j)^2} \right)^2
   \end{matrix}  \right).
\end{equation}
Notice that when $r = 0$ we have that $\l_{2,r} = 0$, and
therefore, as $\sum_{j=1}^l s_j = 1$, the second terms in both
rows are constant, independent
of $\s_2$; a similar consideration holds when $r = 1$.  These
arguments imply that the function $\Phi_\lambda$ is indeed well defined on
$(\gamma_1)_k * (\gamma_2)_l$.

We have then the following result.

\begin{pro}\label{p:en} Suppose $\rho_1 \in (4 k \pi, 4 (k+1) \pi)$ and $\rho_2 \in (4 l \pi, 4 (l+1) \pi)$.
Then one has
$$
  J_{\rho}( \var_{\l , \zeta}) \to - \infty \quad \hbox{ as  } \l \to + \infty
  \qquad \quad \hbox{ uniformly in }  \zeta \in (\gamma_1)_k * (\gamma_2)_l.
$$
\end{pro}

\begin{pf}
We define $v_1,v_2: \Sg \rightarrow \mathbb{R}$ as follows;
\begin{equation*}
    v_1(x) = \log \, \sum_{i=1}^{k} t_i \left( \frac{1}{1 + \l_{1,r}^2 d(x,x_i)^2} \right)^2, \qquad
    v_2(x) = \log \, \sum_{j=1}^{l} s_j \left( \frac{1}{1 + \l_{2,r}^2 d(x,y_j)^2} \right)^2.
\end{equation*}
With this notation the components of $\var(x)$ are given by
$$
  \left( \begin{array}{c} \var_1(x)
   \\ \var_2(x)
  \end{array} \right)
   =  \left(  \begin{matrix}
    v_1(x) - \frac{1}{2} \,v_2(x)
   \\ - \frac{1}{2} \,v_1(x) + v_2(x)
   \end{matrix}  \right).
$$
We first prove two estimates on the gradients of $v_1$ and $v_2$.
\begin{equation} \label{gr1}
    |\n v_i(x)| \leq C \l_{i,r}, \qquad \mbox{for every $x\in\Sg$ and $r\in[0,1],$} \quad i=1,2,
\end{equation}
where $C$ is a constant independent of $\l$, $\zeta \in (\gamma_1)_k * (\gamma_2)_l$, and
\begin{equation} \label{gr2}
    |\n v_i(x)| \leq \frac{4}{d_{\,i,min}(x)}, \qquad \mbox{for every $x\in\Sg,$} \quad i=1,2,
\end{equation}
where $\dis{d_{1,min}(x) = \min_{i=1,\dots,k} d(x,x_i)}$ and $\dis{d_{2,min}(x) = \min_{j=1,\dots,l} d(x,y_j).}$

We show the inequalities just for $v_1$, as for $v_2$ the proof is similar. We have that
$$
    \n v_1(x) = - 2 \l_{1,r}^2 \frac{\sum_{i=1}^k t_i \bigr(1 + \l_{1,r}^2 d^2(x,x_i)\bigr)^{-3} \n \bigr(d^2(x,x_i)\bigr)}{\sum_{j=1}^k t_j \bigr(1 + \l_{1,r}^2 d^2(x,x_j)\bigr)^{-2}}.
$$
Using the estimate $\left|\n \bigr(d^2(x,x_i)\bigr)\right| \leq 2 d(x,x_i)$ and the following inequality
$$
    \frac{\l_{1,r}^2 d(x,x_i)}{1 + \l_{1,r}^2 d^2(x,x_i)} \leq C \l_{1,r}, \qquad i = 1, \dots, k,
$$
with $C$ a fixed constant, we obtain (\ref{gr1}). For proving (\ref{gr2}) we observe that if $\l_{1,r} = 0$ the inequality is trivially satisfied. If instead $\l_{1,r} > 0$ we have
\begin{eqnarray*}
  |\n v_1(x)| & \leq & 4 \l_{1,r}^2 \frac{\sum_{i=1}^k t_i \bigr(1 + \l_{1,r}^2 d^2(x,x_i)\bigr)^{-3} d(x,x_i)}{\sum_{j=1}^k t_j \bigr(1 + \l_{1,r}^2 d^2(x,x_j)\bigr)^{-2}} \leq 4 \l_{1,r}^2 \frac{\sum_{i=1}^k t_i \bigr(1 + \l_{1,r}^2 d^2(x,x_i)\bigr)^{-2} \frac{d(x,x_i)}{\l_{1,r}^2 d^2(x,x_i)}}{\sum_{j=1}^k t_j \bigr(1 + \l_{1,r}^2 d^2(x,x_y)\bigr)^{-2}} \\
   & \leq & 4 \frac{\sum_{i=1}^k t_i \bigr(1 + \l_{1,r}^2 d^2(x,x_i)\bigr)^{-2} \frac{1}{d_{\,1,min}(x)}}{\sum_{j=1}^k t_j \bigr(1 + \l_{1,r}^2 d^2(x,x_j)\bigr)^{-2}} = \frac{4}{d_{\,1,min}(x)},
\end{eqnarray*}
which proves (\ref{gr2}).

We consider now the Dirichlet part of the functional $J_\rho$. Taking into account the definition of $\var_1, \var_2$ we have
\begin{eqnarray*}
    \int_{\Sg} Q(\var_1, \var_2) \,dV_g & = & \frac{1}{3} \int_\Sg \bigr(|\n \var_1|^2 + |\n \var_2|^2 + \n \var_1 \cdot \n \var_2\bigr) \,dV_g \\
                                        & = & \frac{1}{3} \int_\Sg \! \left( \! |\n v_1|^2 + \frac{1}{4}|\n v_2|^2 - \n v_1 \cdot \n v_2 \! \right) \! dV_g + \frac{1}{3} \int_\Sg \! \left( \!|\n v_2|^2 + \frac{1}{4}|\n v_1|^2 - \n v_2 \cdot \n v_1 \! \right) \! dV_g + \\
                                        & + &  \frac{1}{3} \int_\Sg \left(-\frac{1}{2}|\n v_1|^2 - \frac{1}{2}|\n v_2|^2 + \frac54 (\n v_1 \cdot \n v_2)\right) \,dV_g \\
                                        & = & \frac{1}{4} \int_\Sg |\n v_1|^2 \,dV_g + \frac{1}{4} \int_\Sg |\n v_2|^2 \,dV_g - \frac{1}{4} \int_\Sg \n v_1 \cdot\n v_2 \,dV_g.
\end{eqnarray*}
We first observe that the part involving the mixed term $\n v_1 \cdot\n v_2$ is bounded by a constant depending only on $\Sg$. Indeed, we introduce the sets
\begin{equation}\label{sets}
    A_i = \left\{ x \in \Sg : d(x,x_i) = \min_{j=1}^k d(x,x_j) \right\}.
\end{equation}
Using then (\ref{gr2}) we have
\begin{eqnarray*}
    \int_\Sg \n v_1 \cdot\n v_2 \,dV_g & \leq & \int_\Sg |\n v_1| |\n v_2| \,dV_g \leq 16 \int_\Sg \frac{1}{d_{1,min}(x) \, d_{2,min}(x)} \,dV_g(x) \\
                                       & \leq & 16 \sum_{i=1}^k \int_{A_i} \frac{1}{d(x,x_i) \, d_{2,min}(x)} \,dV_g(x).
\end{eqnarray*}
We take now $\d > 0$ such that
$$
    \d = \frac 12 \min \left\{ \min_{i\in \{1,\dots k\}, j\in \{1,\dots l\}} d(x_i,
    y_j),\ \
    \min_{m,\ n\in \{1,\dots k\}, m \neq n} d(x_m,x_n) \right\}
$$
and we split each $A_i$ into $A_i = B_\d(x_i) \cup (A_i \setminus B_\d (x_i)), i = 1,\dots k$. By a change of variables and exploiting the fact that $d_{2,min}(x) \geq \frac 1C$ in $B_\d(x_i)$ we obtain
$$
    \sum_{i=1}^k \int_{B_\d(x_i)} \frac{1}{d(x,x_i) \, d_{2,min}(x)} \,dV_g(x) \leq C.
$$
Using the same argument for the part $A_i \setminus B_\d (x_i)$ with some modifications and exchanging the role of $d_{1,min}$ and $d_{2,min}$ we finally deduce that
\begin{equation} \label{eq:misto}
\int_\Sg \n v_1 \cdot\n v_2 \,dV_g \leq C.
\end{equation}

\noindent We want now to estimate the remaining part of the Dirichlet energy. For convenience we treat
the cases  $r=0$ and $r=1$ separately.  Consider first the case $r=0$: we then have  $\n v_2(x) = 0$ and we get
$$
    \int_{\Sg} Q(\var_1, \var_2) \,dV_g = \frac{1}{4} \int_{\Sg} |\n v_1(x)|^2 \,dV_g(x).
$$
We divide now the integral into two parts;
$$
    \frac{1}{4} \int_{\Sg} |\n v_1(x)|^2 \,dV_g(x) = \frac{1}{4} \int_{\bigcup_i B_{\frac{1}{\l}}(x_i)} |\n v_1(x)|^2 \,dV_g(x) + \frac{1}{4}\int_{\Sg \setminus \bigcup_i B_{\frac{1}{\l}}(x_i)} |\n v_1(x)|^2 \,dV_g(x).
$$
From (\ref{gr1}) we deduce that
$$
    \int_{\bigcup_i B_{\frac{1}{\l}}(x_i)} |\n v_1(x)|^2 \,dV_g(x) \leq C.
$$
Using then (\ref{gr2}) for the second part of the integral, recalling the definition (\ref{sets}) of the sets $A_i$, one finds that
\begin{eqnarray*}
    \frac 14 \int_{\Sg \setminus \bigcup_i B_{\frac{1}{\l}}(x_i)} |\n v_1(x)|^2(x) \,dV_g  & \leq & 4 \int_{\Sg \setminus \bigcup_i B_{\frac{1}{\l}}(x_i)} \frac{1}{d_{1,min}^2(x)} \,dV_g(x) + C \\
    & \leq & 4 \sum_{i=1}^k \int_{A_i \setminus B_{\frac{1}{\l}}(x_i)} \frac{1}{d_{1,min}^2(x)} \,dV_g(x) + C \\
                                         & \leq & 8 k \pi\bigr(1 + o_\l(1)\bigr) \log \l + C,
\end{eqnarray*}
where $o_\l(1) \to 0$ as  $\l \to + \infty$. Therefore we have
\begin{equation} \label{eq:1111111}
\int_{\Sg} Q(\var_1, \var_2) \,dV_g \leq 8 k \pi\bigr(1 + o_\l(1)\bigr) \log \l + C.
\end{equation}
Reasoning as in \cite{mreview}, Proposition 4.2 part $(ii)$, it is possible to show that
$$
  \int_\Sg v_1 \,dV_g = - 4 \bigr(1 + o_\l(1)\bigr) \log \l;  \qquad\qquad \log \int_\Sg e^{v_1} \,dV_g = - 2 \bigr(1 + o_\l(1)\bigr) \log \l$$
$$
  \log \int_\Sg e^{-\frac {1}{2} v_1} \,dV_g = 2 \bigr(1 + o_\l(1)\bigr) \log \l,
$$
and clearly
$$
\int_\Sg v_2 \,dV_g = O(1);  \qquad \log \int_\Sg e^{v_2} \,dV_g = O(1); \qquad \log \int_\Sg e^{-\frac {1}{2} v_2} \,dV_g = O(1).
$$
Therefore we get
$$
  \int_\Sg \var_1 \,dV_g = - 4 \bigr(1 + o_\l(1)\bigr) \log \l;  \qquad \log \int_\Sg e^{\var_1} \,dV_g = - 2 \bigr(1 + o_\l(1)\bigr) \log \l;
$$
$$
    \int_\Sg \var_2 \,dV_g = 2 (1 + o_\l(1)) \log \l;  \qquad \log \int_\Sg e^{\var_2} \,dV_g = - 2 (1 + o_\l(1)) \log \l.
$$
Inserting the latter equalities in the expression of the functional $J_\rho$ and using the fact that $\widetilde{h}_i \geq \frac{1}{C}, i=1,2$ outside a small neighbourhood of the singular points (which are avoided by the curves $\gamma_1, \gamma_2$), we obtain
$$
  J_\rho(\var_1, \var_2) \leq \bigr( 8 k \pi - 2 \rho_1 + o_\l(1) \bigr)\log \l + C,
$$
where $C$ is independent of $\l$ and $\s_1 ,\s_2$.

For the case $r=1$, by the same argument we have that
$$
  J_\rho(\var_1, \var_2) \leq \bigr( 8 l \pi - 2 \rho_2 + o_\l(1) \bigr)\log \l + C.
$$

\noindent We consider now the case $r\in (0,1)$. By \eqref{eq:misto} the Dirichlet part can be estimated by
$$
    \int_{\Sg} Q(\var_1, \var_2) \,dV_g \leq \frac{1}{4} \int_{\Sg} |\n v_1(x)|^2 \,dV_g(x) + \frac{1}{4} \int_{\Sg} |\n v_2(x)|^2 \,dV_g(x) + C.
$$
For general $r$ one can just substitute $\l$ with $\l_{1,r}$ in \eqref{eq:1111111} (and similarly for the $v_2$),
to get the following estimate
\begin{equation} \label{gr}
    \int_{\Sg} Q(\var_1, \var_2) \,dV_g \leq 8 k \pi\bigr(1 + o_\l(1)\bigr) \log \bigr(\l_{1,r} + \d_{1,r}\bigr) + 8 l \pi\bigr(1 + o_\l(1)\bigr) \log \bigr(\l_{2,r} + \d_{2,r}\bigr) + C,
\end{equation}
where $\d_{1,r} > \d > 0$ as $r \to 1$ and $\d_{2,r} > \d > 0$ as $r \to 0$, for
some fixed $\d$. The same argument as for $r=0,1$ leads to
$$
  \int_\Sg v_1 \,dV_g = - 4 \bigr(1 + o_\l(1)\bigr) \log \bigr(\l_{1,r} + \d_{1,r}\bigr) + O(1); \qquad \int_\Sg v_2 \,dV_g = - 4 \bigr(1 + o_\l(1)\bigr) \log \bigr(\l_{2,r} + \d_{2,r}\bigr) + O(1),
$$
therefore we obtain
\begin{eqnarray}
    \int_\Sg \var_1 \,dV_g = - 4 \bigr(1 + o_\l(1)\bigr) \log \bigr(\l_{1,r} + \d_{1,r}\bigr) + 2 \bigr(1 + o_\l(1)\bigr) \log \bigr(\l_{2,r} + \d_{2,r}\bigr) + O(1), \label{av1} \\
    \int_\Sg \var_2 \,dV_g = 2 \bigr(1 + o_\l(1)\bigr) \log \bigr(\l_{1,r} + \d_{1,r}\bigr) - 4 \bigr(1 + o_\l(1)\bigr) \log \bigr(\l_{2,r} + \d_{2,r}\bigr) + O(1).   \label{av2}
\end{eqnarray}
We consider now the exponential term. We have
$$
    \int_\Sg e^{\var_1} \,dV_g = \sum_{i=1}^k t_i \int_\Sg \frac{1}{\bigr( 1 + \l_{1,r}^2 d(x,x_i)^2 \bigr)^2} \left( \sum_{j=1}^l s_j \frac{1}{\bigr( 1 + \l_{2,r}^2 d(x,y_j)^2 \bigr)^2} \right)^{-\frac 12}\,dV_g(x).
$$

Clearly it is enough to estimate the term
$$
\int_\Sg \frac{1}{\bigr( 1 + \l_{1,r}^2 d(x,\ov{x})^2 \bigr)^2} \left( \sum_{j=1}^l s_j \frac{1}{\bigr( 1 + \l_{2,r}^2 d(x,y_j)^2 \bigr)^2} \right)^{-\frac 12}\,dV_g(x)
$$
with $\ov{x}\in \{ x_1,\dots x_k \}$. Letting $\dis{\d = \frac{\min_j \{ d(\ov{x},y_j) \}}2}$ we  divide the domain into two regions as follows: $\Sg = B_\d(\ov{x}) \cup (\Sg \setminus B_\d(\ov{x}))$. When we integrate in $B_\d(\ov{x})$ we perform a change of variables for the part involving $\l_{1,r}$ and observing that $\frac 1C \leq d(x, y_j) \leq C, j=1,\dots,l$, for every $x\in B_\d(\ov{x})$, we deduce
$$
\int_{B_\d(\ov{x})} \frac{1}{\bigr( 1 + \l_{1,r}^2 d(x,\ov{x})^2 \bigr)^2} \left( \sum_{j=1}^l s_j \frac{1}{\bigr( 1 + \l_{2,r}^2 d(x,y_j)^2 \bigr)^2} \right)^{-\frac 12}\,dV_g(x) = \frac{\bigr(\l_{2,r} + \d_{2,r}\bigr)^2}{\bigr(\l_{1,r} + \d_{1,r}\bigr)^2}\bigr(1 + O(1)\bigr).
$$
On the other hand for the integral over $\Sg \setminus
B_\d(\ov{x})$ we use that $\frac 1C \leq d(x, \ov{x}) \leq C$ to
get that this part is a higher-order term and can be absorbed by
the latter estimate. Recall now that $\wtilde{h}_1$ stays bounded
away from zero in a neighbourhood of the curve $\gamma_1$ (see
the beginning of the section). Therefore, since the contribution
of the integral outside a neighbourhood of $\gamma_1$ is
negligible, we can conclude that
\begin{equation} \label{exp1}
\log \int_\Sg  \wtilde{h}_1 e^{\var_1} \,dV_g = 2 \log \bigr(\l_{2,r} + \d_{2,r}\bigr) - 2 \log \bigr(\l_{1,r} + \d_{1,r}\bigr) + O(1).
\end{equation}
Similarly we have that
\begin{equation} \label{exp2}
\log \int_\Sg \wtilde{h}_2 e^{\var_2} \,dV_g = 2\log \bigr(\l_{1,r} + \d_{1,r}\bigr) - 2\log \bigr(\l_{2,r} + \d_{2,r}\bigr) + O(1).
\end{equation}

\noindent Using the estimates (\ref{gr}), (\ref{av1}), (\ref{av2}), (\ref{exp1}) and (\ref{exp2}) we finally obtain
$$
  J_\rho(\var_1, \var_2) \leq \bigr( 8 k \pi - 2 \rho_1 + o_\l(1) \bigr)\log \bigr(\l_{1,r} + \d_{1,r}\bigr) + \bigr( 8 l \pi - 2 \rho_2 + o_\l(1) \bigr)\log \bigr(\l_{2,r} + \d_{2,r}\bigr) + O(1).
$$
Recalling that $\rho_1 > 4k\pi, \rho_2 > 4l\pi$ and observing that $\dis{\max_{r\in[0,1]}\{ \l_{1,r}, \l_{2,r} \}} \to +\infty$ as $\l \to \infty$, we conclude the proof.
\end{pf}

\section{Moser-Trudinger inequalities and topological join} \label{s:mtineq}

\noindent In this section we are going to give an improved version of the Moser-Trudinger inequality $\dis{\eqref{mtjw}}$, where the constant $\dis{4\pi}$ can be replaced by an integer multiple under the assumption that the integral of $\dis{\wtilde h_ie^{u_i}}$ is distributed on different sets with positive mutual distance. The improved inequality implies that if $\dis{J_\rh(u_1,u_2)}$ attains very low values, then $\dis{\wtilde h_ie^{u_i}}$ has to concentrate near a given number (depending on $\rho_i$) of points
for some $\dis{i\in\{1,2\}}$. As anticipated in the introduction, we will see that this induces
a natural map  from low sub-levels of $J_\rho$ to the topological join of some sets of barycenters. This   extends some analysis  from \cite{jlw} and \cite{cheikh}, where the authors considered the case $\dis{\rh_2<4\pi}$, and from \cite{mr2}, where both parameters belong to the range $\dis{(4\pi,8\pi)}$. We start with a covering lemma:

\begin{lem}\label{l:step1}
Let $\dis{\d>0}$, $\dis{\th>0}$, $\dis{k,l\in\N}$ with $\dis{k\ge l}$, $\dis{f_i\in L^1(\Sg)}$ be non-negative functions with $\dis{\|f_i\|_{L^1(\Sg)}=1}$ for $\dis{i=1,2}$ and $\dis{\{\O_{1,i},\O_{2,j}\}_{i\in\{0,\dots,k\},j\in\{0,\dots,l\}}\subset\Sigma}$ such that
$$d(\O_{1,i},\O_{1,i'})\ge\d\qquad\qquad\forall\;i, \ i'\in\{0,\dots,k\}\hbox{ with }i\ne i';$$
$$d(\O_{2,j},\O_{2,j'})\ge\d\qquad\qquad\forall\;j,\ j' \in\{0,\dots,l\}\hbox{ with }j\ne j',$$
and
$$\int_{\O_{1,i}}f_1dV_g\ge\th\qquad\qquad\forall\;i\in\{0,\dots,k\};$$
$$\int_{\O_{2,j}}f_2dV_g\ge\th\qquad\qquad\forall\;j\in\{0,\dots,l\}.$$
Then, there exist $\dis{\ov\d>0,\;\ov\th>0}$, independent of
$\dis{f_i}$, and $\dis{\{\O_n\}_{n=0}^k\subset\Sigma}$ such that
$$d(\O_n,\O_n')\ge\ov\d\qquad\qquad\forall\;n,\ n'\in\{0,\dots,k\}\hbox{ with }n\ne n'$$
and
$$ |\O_n| \geq \ov\th \qquad\qquad\forall\;n\in\{0,\dots,k\};$$
$$\int_{\O_n}f_1dV_g\ge\ov\th\qquad\qquad\forall\;n\in\{0,\dots,k\};$$
$$\int_{\O_n}f_2dV_g\ge\ov\th\qquad\qquad\forall\;n\in\{0,\dots,l\}.$$
\end{lem}

\begin{pf}
We set $\dis{\ov\d=\frac{\d}8}$ and consider the open cover
$\dis{\left\{B_{\ov\d}(x)\right\}_{x\in\Sg}}$ of $\dis{\Sg}$; by compactness,
$\dis{\Sg\subset\bigcup_{h=1}^{H}B_{\ov\d}(x_h)}$ for some $\dis{\{x_h\}_{h=1}^{H}\subset\Sg}$,
$H=H\left(\ov\d,\Sg\right)$.\\
We choose
$\dis{\{y_{1,i},y_{2,j}\}_{i\in\{0,\dots,k\},j\in\{0,\dots,l\}}\subset\{x_h\}_{h=1}^{H}}$
such that
$$\int_{B_{\ov\d}(y_{1,i})}f_1dV_g=\max\left\{\int_{B_{\ov\d}(x_h)}f_1dV_g:\;B_{\ov\d}(x_h)\cap\O_{1,i}\ne\emptyset\right\};$$
$$\int_{B_{\ov\d}(y_{2,j})}f_2dV_g=\max\left\{\int_{B_{\ov\d}(x_h)}f_2dV_g:\;B_{\ov\d}(x_h)\cap\O_{2,j}\ne\emptyset\right\}$$
Since $d(y_{1,i}, \Omega_{1,i})< \ov\d$,  we have that
$\dis{d(y_{1,i},y_{1,i'})\ge6\ov\d}$ for $i\neq i'$. Analogously,
$\dis{d(y_{2,j},y_{2,j'})\ge6\ov\d}$ if $j \neq j'$.\\
In particular, this implies that for any $\dis{i\in\{0,\dots,k\}}$
there exists at most one $\dis{j(i)}$ such that
$\dis{d(y_{2,j(i)},y_{1,i})<3\ov\d}$. We relabel the index $i$ so
that for $i=1,\  \dots l$ such $j(i)$ exists, and we relabel the
index $j$ so that $j(i)=i$. We now define:
$$\O_n:=\left\{\begin{array}{ll}B_{\ov\d}(y_{1,n})\cup B_{\ov\d}(y_{2,n})&\hbox{if }n\in\{0,\dots,l\}\\
B_{\ov\d}(y_{1,n})&\hbox{if
}n\in\{l+1,\dots,k\}.\end{array}\right.
$$ In other words, we make unions of balls $B_{\ov\d}(y_{1,n})\cup
B_{\ov\d}(y_{2,n})$ if they are close to each other: for separate
balls, we make arbitrary unions. If $k>l$, the remaining balls are
considered alone.

It is easy to check that those sets satisfy the theses of Lemma
\ref{l:step1}.

\end{pf}

\

\noindent To show the improved Moser-Trudinger inequality, we will need a {\em localized}
version of the inequality $\dis{\eqref{mtjw}}$, which was proved in \cite{mr2}.

\begin{lem}\label{l:mrtoda} (\cite{mr2})
Let $\dis{\d>0}$ and $\dis{\O\Subset\wtilde\O\subset\Sg}$ be such that $\dis{d\left(\O, \partial \wtilde\O\right)\ge\d}$.\\
Then, for any $\dis{\e>0}$ there exists $\dis{C=C(\e,\d)}$ such that for any $\dis{u=(u_1,u_2)\in H^1(\Sg)\times H^1(\Sg)}$
$$\log\int_\O e^{u_1-\fint_{\wtilde\O}u_1dV_g}dV_g+\log\int_\O e^{u_2-\fint_{\wtilde\O}u_2dV_g}dV_g\le
\frac{1}{4\pi}\int_{\wtilde\O}Q(u_1,u_2)dV_g+
\e\int_{\Sigma}Q(u_1,u_2)dV_g+  C.$$
\end{lem}

\noindent Here comes the improved inequality: basically, if the {\em mass} of both $\dis{\wtilde h_1e^{u_1}}$ and $\dis{\wtilde h_2e^{u_2}}$ is spread respectively on at least $\dis{k+1}$ and $\dis{l+1}$ different sets, then the logarithms in $\dis{\eqref{mtjw}}$ can be multiplied by $\dis{k+1}$ and $\dis{l+1}$ respectively.\\
Notice that this result was given in \cite{cheikh} in the case $\dis{l=0}$ and in \cite{mr2} in the case $\dis{k=l=1}$.

\begin{lem}\label{l:imprc}
Let $\dis{\d>0}$, $\dis{\th>0}$, $\dis{k,l\in\N}$ and
$\dis{\{\O_{1,i},\O
_{2,j}\}_{i\in\{0,\dots,k\},j\in\{0,\dots,l\}}\subset\Sigma}$ be
such that
$$d(\O_{1,i},\O_{1,i'})\ge\d\qquad\qquad\forall\;i,\ i'\in\{0,\dots,k\}\hbox{ with }i\ne i';$$
$$d(\O_{2,j},\O_{2,j'})\ge\d\qquad\qquad\forall\;j,\ j'\in\{0,\dots,l\}\hbox{ with }j\ne j'.$$
Then, for any $\dis{\e>0}$ there exists $\dis{C=C\left(\e,\d,\th,k,l,\Sg\right)}$ such that any $\dis{u=(u_1,u_2)\in H^1(\Sg)\times H^1(\Sg)}$ satisfying
$$\int_{\O_{1,i}}\wtilde h_1e^{u_1}dV_g\ge\th\int_\Sg\wtilde h_1e^{u_1}dV_g\qquad\qquad\forall\;i\in\{0,\dots,k\};$$
$$\int_{\O_{2,j}}\wtilde h_2e^{u_2}dV_g\ge\th\int_\Sg\wtilde h_2e^{u_2}dV_g\qquad\qquad\forall\;j\in\{0,\dots,l\}$$
verifies
$$(k+1)\log\int_\Sg\wtilde h_1e^{u_1-\ov{u_1}}dV_g+(l+1)\log\int_\Sg\wtilde h_2e^{u_2-\ov{u_2}}dV_g\le \frac{1+\e}{4\pi}\int_\Sg Q(u_1,u_2)dV_g+C.$$
\end{lem}

\begin{pf} In the proof we assume that $\ov{u_1}=\ov{u_2}=0$.
After relabelling the indexes, we can suppose $\dis{k\ge l}$ and apply Lemma $\dis{\ref{l:step1}}$ with $\dis{f_i=\frac{\wtilde h_ie^{u_i}}{\int_\Sg\wtilde h_ie^{u_i}dV_g}}$ to get $\dis{\{\O_j\}_{j=0}^k\subset\Sigma}$ with
$$d(\O_i,\O_j)\ge\ov\d\qquad\qquad\forall\;i,j\in\{0,\dots,k\}\hbox{ with }i\ne j$$
and
$$\int_{\O_i}\wtilde h_1e^{u_1}dV_g\ge\ov\th\int_\Sg\wtilde h_1e^{u_1}dV_g\qquad\qquad\forall\;i\in\{0,\dots,k\};$$
$$\int_{\O_j}\wtilde h_2e^{u_2}dV_g\ge\ov\th\int_\Sg\wtilde h_2e^{u_2}dV_g\qquad\qquad\forall\;j\in\{0,\dots,l\}.$$

Notice that:
$$\log\int_\Sg\wtilde h_ie^{u_i}dV_g=\fint_{\wtilde\O_j} u_i dV_g+\log\int_\Sg\wtilde h_1e^{u_i-\fint_{\wtilde\O_j}u_idV_g}dV_g, \ i=1,2.$$

The average on $\wtilde\O_j$ can be estimated by Poincar\'{e} inequality:

\begin{eqnarray}
&&\fint_{\wtilde\O_j} u_i dV_g \le \frac{1}{\left|\wtilde\O_j\right|} \int_{\Sg} |u_i| dV_g \le C \left ( \int_{\Sg} |\nabla u_i|^2 dV_g  \right )^{1/2} \le C+ \e  \int_{\Sg} |\nabla u_i|^2 dV_g,\ i=1,\ 2.
\end{eqnarray}

We now apply, for any $\dis{j\in\{0,\dots,k\}}$ Lemma $\dis{\ref{l:mrtoda}}$ with , $\dis{\O=\O_j}$ and $\dis{\wtilde\O=\wtilde\O_j:=\left\{x \in \Sigma: d(x,\O_j)<\frac{\ov \d}{2}\right\}}
$: for $\dis{j\in\{0,\dots,l\}}$ we get

\begin{eqnarray} \label{caso1}
&&\log\int_\Sg\wtilde
h_1e^{u_1-\fint_{\wtilde\O_j}u_1dV_g}dV_g+\log\int_\Sg\wtilde
h_2e^{u_2-\fint_{\wtilde\O_j}u_2dV_g}dV_g\\\nonumber &\le&
2\log\frac{1}{\ov\th}+\log\int_{\O_j}\wtilde
h_1e^{u_1-\fint_{\wtilde\O_j}u_1dV_g}dV_g+\log\int_{\O_j}\wtilde
h_2e^{u_2-\fint_{\wtilde\O_j}u_2dV_g}dV_g\\\nonumber &\le& C+
\log\int_{\O_j}e^{u_1-\fint_{\wtilde\O_j}u_1dV_g}dV_g+\log\int_{\O_j}e^{u_2-\fint_{\wtilde\O_j}u_2dV_g}dV_g\\\nonumber
\label{eq:l0} &\le&
C+\frac{1}{4\pi}\int_{\wtilde\O_j}Q(u_1,u_2)dV_g + \e
\int_{\Sigma}Q(u_1,u_2)dV_g, \ j=1, \dots l.
\end{eqnarray}
For $\dis{j\in\{l+1,\dots,k\}}$ we have
\begin{eqnarray}
&& \log\int_\Sg\wtilde h_1e^{u_1-\fint_{\wtilde\O_j}
u_1dV_g}dV_g\le\log\frac{1}{\ov\th}+\left\|\wtilde
h_1\right\|_{L^\infty(\Sg)}+\log\int_{\O_j}e^{u_1-\fint_{\wtilde\O_j}u_1dV_g}dV_g\\\nonumber
\label{eq:l1} &\le&
C-\log\int_{\O_j}e^{u_2-\fint_{\wtilde\O_j}u_2dV_g}dV_g+\frac{1}{4\pi}\int_{\wtilde\O_j}Q(u_1,u_2)dV_g
+ \e \int_{\Sigma} Q(u_1,u_2)dV_g.
\end{eqnarray}

The exponential term on the second component can be estimated  by
using Jensen's inequality:
\begin{eqnarray}
&&\log\int_{\O_j}e^{u_2-\fint_{\wtilde\O_j}u_2dV_g}dV_g=\log|\O_j|+\log\fint_{\O_j}e^{u_2-\fint_{\wtilde\O_j}u_2dV_g}dV_g \\\nonumber
\label{eq:l3}
&\ge&\log|\O_j|\ge -C.
\end{eqnarray}

Putting together \eqref{eq:l1} and \eqref{eq:l3}, we have:
\begin{equation} \label{caso2} \log\int_\Sg\wtilde h_1e^{u_1-\fint_{\wtilde\O_j}u_1dV_g}dV_g\le
\frac{1}{4\pi}\int_{\wtilde\O_j}Q(u_1,u_2)dV_g + \e \int_{\Sigma}
Q(u_1,u_2)dV_g + C, \ j= l+1 \dots k. \end{equation} Summing over
all $\dis{j\in\{0,\dots,k\}}$ and taking into account
\eqref{caso1}, \eqref{caso2}, we obtain the result, renaming $\e$
appropriately. \end{pf}

\noindent We will now use a technical result that gives sufficient conditions to apply Lemma $\dis{\ref{l:imprc}}$. Its proof can be found for instance in \cite{dja,cheikh}.

\begin{lem}\label{l:er} (\cite{cheikh,mr2})
Let $\dis{f\in L^1(\Sg)}$ be a non-negative function with $\dis{\|f\|_{L^1(\Sg)}=1}$ and let $\dis{m\in\N}$ be such that there exist $\dis{\e>0,\;s>0}$ with
$$\int_{\bigcup_{j=0}^mB_s(x_j)}fdV_g<1-\e\qquad\qquad\forall\;\{x_j\}_{j=0}^m\subset\Sg.$$
Then there exist $\dis{\ov\e>0,\;\ov s>0}$, not depending on $\dis{f}$, and $\dis{\left\{\ov x_j\right\}_{j=1}^m\subset\Sg}$ satisfying
$$\int_{B_{\ov s}(\ov x_j)}fdV_g>\ov{\e}\qquad\qquad\forall\;j\in\{1,\dots,m\},$$
$$B_{2\ov s}\left(\ov x_i\right)\cap B_{2\ov s}\left(\ov x_j\right)=\emptyset\qquad\qquad\forall\;i,j\in\{1,\dots,m\},i\ne j.$$
\end{lem}

\

Now we have enough tools to obtain information on the structure of very low sub-levels of $\dis{J_\rh}$:
\begin{lem}\label{l:II<-M}
Suppose $\dis{\rh_1\in(4k\pi,4(k+1)\pi)}$ and $\dis{\rh_2\in(4l\pi,4(l+1)\pi)}$. Then, for any $\dis{\e>0,\;s>0}$, there exists $\dis{L=L(\e,s)>0}$ such that for any $\dis{u\in J_\rh^{-L}}$ there are either some $\dis{\{x_i\}_{i=1}^k\subset\Sg}$ verifying
$$\frac{\int_{\bigcup_{i=1}^kB_s(x_i)}\wtilde h_1e^{u_1}dV_g}{\int_\Sg\wtilde h_1e^{u_1}dV_g}\ge1-\e$$
or some $\dis{\{y_j\}_{j=1}^l\subset\Sg}$ verifying
$$\frac{\int_{\bigcup_{j=1}^lB_s(y_j)}\wtilde h_2e^{u_2}dV_g}{\int_\Sg\wtilde h_2e^{u_2}dV_g}\ge1-\e.$$
\end{lem}

\begin{pf}
Suppose by contradiction that the statement is not true, that is
there are $\dis{\e_1,\e_2>0,\;s_1,s_2>0}$, and
$\dis{\{u_n=(u_{1,n},u_{2,n})\}_{n\in\N}\subset H^1(\Sg)\times
H^1(\Sg)}$ such that
$\dis{J_\rh(u_{1,n},u_{2,n})\underset{n\to+\infty}\longrightarrow
-\infty}$ and
$$\frac{\int_{\bigcup_{i=1}^kB_{s_1}\left(x_i\right)}\wtilde h_1e^{u_{1,n}}dV_g}{\int_\Sg\wtilde h_1e^{u_{1,n}}dV_g}<1-\e_1;\qquad\frac{\int_{\bigcup_{j=1}^lB_{s_2}\left(y_j\right)}\wtilde h_2e^{u_{2,n}}dV_g}{\int_\Sg\wtilde h_2e^{u_{2,n}}dV_g}<1-\e_2,\qquad\forall\;\{x_i\}_{i=1}^k,\{y_j\}_{j=1}^l\subset\Sg.$$
Then, we may apply twice Lemma $\dis{\ref{l:er}}$ with $\dis{f=\frac{\wtilde h_ie^{u_i}}{\int_\Sg\wtilde h_ie^{u_i}dV_g}}$, $\dis{\wtilde\e=\e_i}$, $\dis{\wtilde s=s_i}$ and find $\dis{\ov\e_1,\ov\e_2>0,\;\ov s_1,\ov s_2>0}$ and $\dis{\{\ov x_i\}_{i=0}^k,\{\ov y_j\}_{j=0}^l}$ with
$$\int_{B_{\ov s_1}\left(\ov x_i\right)}\wtilde h_1e^{u_1}dV_g\ge\ov\e_1\int_\Sg\wtilde h_1e^{u_1}dV_g\qquad\qquad\forall\;i\in\{0,\dots,k\};$$
$$\int_{B_{\ov s_2}\left(\ov y_j\right)}\wtilde h_2e^{u_2}dV_g\ge\ov\e_2\int_\Sg\wtilde h_2e^{u_2}dV_g\qquad\qquad\forall\;j\in\{0,\dots,l\},$$
and
$$B_{2\ov s_1}\left(\ov x_i\right)\cap B_{2\ov s_1}\left(\ov x_j\right)=\emptyset\qquad\qquad\forall\;i,j\in\{0,\dots,k\}\hbox{ with }i\ne j;$$
$$B_{2\ov s_2}\left(\ov y_j\right)\cap B_{2\ov s_2}\left(\ov y_j\right)=\emptyset\qquad\qquad\forall\;i,j\in\{0,\dots,l\}\hbox{ with }i\ne j.$$
Hence, we obtain an improved Moser-Trudinger inequality for $\dis{u_n=(u_{1,n},u_{2,n})}$ applying Lemma $\dis{\ref{l:imprc}}$ with $\dis{\wtilde\d:=2\min\{\ov s_1,\ov s_2\}}$, $\dis{\wtilde\th:=\min\{\ov\e_1,\ov\e_2\}}$ and $\dis{\O_{1,i}:=B_{\ov s_1}\left(\ov x_i\right)}$, $\dis{\O_{2,j}:=B_{\ov s_2}\left(\ov y_j\right)}$.\\
Moreover, Jensen's inequality gives
$$\int_\Sg\wtilde h_ie^{u_{i,n}-\ov{u_{i,n}}}dV_g=\int_\Sg e^{\log\wtilde h_i+u_{i,n}-\ov{u_{i,n}}}dV_g\ge e^{\int_\Sg\log\wtilde h_idV_g},$$
so, choosing
$$\wtilde\e\in\left(0,\min\left\{\frac{4\pi(k+1)}{\rh_1}-1,\frac{4\pi(l+1)}{\rh_2}-1\right\}\right)$$
we get
\begin{eqnarray*}
-\infty&\underset{n\to+\infty}\longleftarrow& J_\rh(u_{1,n},u_{2,n})\\
&\ge&\left(\frac{4\pi(k+1)}{1+\wtilde\e}-\rh_1\right)\log\int_\Sg\wtilde h_1e^{u_{1,n}-\ov{u_{1,n}}}dV_g+\left(\frac{4\pi(l+1)}{1+\wtilde\e}-\rh_2\right)\log\int_\Sg\wtilde h_2e^{u_{2,n}-\ov{u_{2,n}}}dV_g-C\\
&\ge&\left(\frac{4\pi(k+1)}{1+\wtilde\e}-\rh_1\right)\int_\Sg\log\wtilde h_1dV_g+\left(\frac{4\pi(l+1)}{1+\wtilde\e}-\rh_2\right)\int_\Sg\log\wtilde h_2dV_g-C\\
&\ge&-C
\end{eqnarray*}
that is a contradiction.
\end{pf}

\

\noindent An immediate consequence of the previous Lemma is that at least one of the two $\dis{\wtilde h_ie^{u_i}}$'s
(once normalized in $L^1$) has to be very close respectively to the sets of $\dis{k}$-barycenters or $\dis{l}$-barycenters over $\dis{\Sg}$:
\begin{pro}\label{p:altern}
Suppose $\dis{\rh_1\in(4k\pi,4(k+1)\pi)}$ and
$\dis{\rh_2\in(4l\pi,4(l+1)\pi)}$. Then, for any
$\displaystyle{\e>0}$, there exists $\dis{L>0}$ such that any
$\dis{u\in J_\rh^{-L}}$ verifies either $$\dkr\left(\frac{\wtilde
h_1e^{u_1}}{\int_\Sg\wtilde
h_1e^{u_1}dV_g},\Sg_k\right)<\e\qquad\qquad\text{or}\qquad\qquad
\dkr\left(\frac{\wtilde h_2e^{u_2}}{\int_\Sg\wtilde
h_2e^{u_2}dV_g},\Sg_l\right)<\e.$$

\end{pro}
\begin{pf}
We apply Lemma $\dis{\ref{l:II<-M}}$ with $\dis{\wtilde\e=\frac{\e}4,\;\wtilde s=\frac{\e}2}$; it is not restrictive to suppose that the first alternative occurs and that $\dis{\int_\Sg\wtilde h_1e^{u_1}dV_g=1}$. Hence we get $\dis{L}$ and $\dis{\{x_i\}_{i=1}^k}$ and we define, for such an $\dis{u=(u_1,u_2)\in J_\rho^{-L}}$,
$$\s_1(u)=\sum_{i=1}^{k}t_i(u)\d_{x_i}\in\Sg_k\qquad\text{where }t_i(u)=\int_{B_{\wtilde s}(x_i)\backslash\bigcup_{j=1}^{i-1}B_{\wtilde s}(x_j)}\wtilde h_1e^{u_1}dV_g+\frac{1}{k}\int_{\Sg\backslash\bigcup_{j=1}^kB_{\wtilde s}(x_j)}\wtilde h_1e^{u_1}dV_g.$$
Then, for any $\dis{\phi\in Lip(\Sg)}$,
\begin{eqnarray*}
&& \left|\int_{\Sg\backslash\bigcup_{i=1}^kB_{\wtilde s}(x_i)}\left(\frac{\wtilde h_1e^{u_1}}{\int_\Sg\wtilde h_1e^{u_1}dV_g}-\s_1(u)\right)\phi dV_g\right|=\\
&=&\int_{\Sg\backslash\bigcup_{i=1}^kB_{\wtilde s}(x_i)}\wtilde h_1e^{u_1}\phi dV_g\le\int_{\Sg\backslash\bigcup_{i=1}^kB_{\wtilde s}(x_i)}\wtilde h_1e^{u_1}dV_g\|\phi\|_{L^\infty(\Sg)}<\wtilde\e\|\phi\|_{L^\infty(\Sg)}
\end{eqnarray*}
and
\begin{eqnarray*}
&&\left|\int_{\bigcup_{i=1}^kB_{\wtilde s}(x_i)}\left(\frac{\wtilde h_1e^{u_1}}{\int_\Sg\wtilde h_1e^{u_1}dV_g}-\s_1(u)\right)\phi dV_g\right|\\
&=&\left|\int_{\bigcup_{i=1}^kB_{\wtilde s}(x_i)}\wtilde h_1e^{u_1}\phi dV_g-\sum_{i=1}^k\left(\int_{B_{\wtilde s}(x_i)\backslash\bigcup_{j=1}^{i-1}B_{\wtilde s}(x_j)}\wtilde h_1e^{u_1}dV_g+\frac{1}k\int_{\Sg\backslash\bigcup_{j=1}^kB_{\wtilde s}(x_j)}\wtilde h_1e^{u_1}dV_g\right)\phi(x_i)\right|\\
&=&\left|\int_{\bigcup_{i=1}^k\left(B_{\wtilde s}(x_i)\backslash\bigcup_{j=1}^{i-1}B_{\wtilde s}(x_j)\right)}\wtilde h_1e^{u_1}(\phi-\phi(x_i))dV_g-\int_{\Sg\backslash\bigcup_{j=1}^kB_{\wtilde s}(x_j)}\wtilde h_1e^{u_1}dV_g\phi(x_i)\right|\\
&\le&\wtilde s\|\n\phi\|_{L^\infty(\Sg)}\int_{\bigcup_{i=1}^k\left(B_{\wtilde s}(x_i)\backslash\bigcup_{j=1}^{i-1}B_{\wtilde s}(x_j)\right)}\wtilde h_1e^{u_1}dV_g+\|\phi\|_{L^\infty(\Sg)}\int_{\Sg\backslash\bigcup_{j=1}^kB_{\wtilde s}(x_j)}\wtilde h_1e^{u_1}dV_g\\
&<&\wtilde s\|\n\phi\|_{L^\infty(\Sg)}+\wtilde\e\|\phi\|_{L^\infty(\Sg)}.
\end{eqnarray*}
Hence we can conclude the proof:
\begin{eqnarray*}
&& \dkr\left(\frac{\wtilde h_1e^{u_1}}{\int_\Sg\wtilde h_1e^{u_1}dV_g},\Sg_k\right)\le \dkr\left(\frac{\wtilde h_1e^{u_1}}{\int_\Sg\wtilde h_1e^{u_1}dV_g},\s_1(u)\right)=\sup_{\|\phi\|_{Lip(\Sg)}=1}\left|\int_\Sg\left(\frac{\wtilde h_1e^{u_1}}{\int_\Sg\wtilde h_1e^{u_1}dV_g}-\s_1(u)\right)\phi dV_g\right|\\
&=&\sup_{\|\phi\|_{Lip(\Sg)}=1}\left|\int_{\Sg\backslash\bigcup_{i=1}^kB_{\wtilde s}(x_i)}\left(\frac{\wtilde h_1e^{u_1}}{\int_\Sg\wtilde h_1e^{u_1}dV_g}-\s_1(u)\right)\phi dV_g\right|\\
&+&\sup_{\|\phi\|_{Lip(\Sg)}=1}\left|\int_{\bigcup_{i=1}^kB_{\wtilde s}(x_i)}\left(\frac{\wtilde h_1e^{u_1}}{\int_\Sg\wtilde h_1e^{u_1}dV_g}-\s_1(u)\right)\phi dV_g\right|\\
&<&\sup_{\|\phi\|_{Lip(\Sg)}=1}2\wtilde\e\|\phi\|_{L^\infty(\Sg)}+\wtilde s\|\n\phi\|_{L^\infty(\Sg)}\le2\wtilde\e+\wtilde s=\e,
\end{eqnarray*}
as desired.
\end{pf}

\

With the previous estimates, it is now easy to define a projection
map in the following form:

\begin{pro}\label{p:proj}
Suppose $\rho_1 \in (4 k \pi, 4 (k+1) \pi)$, $\rho_2 \in (4 l \pi, 4 (l+1) \pi)$ and let $\Phi_\lambda$ be as in $\eqref{e:pl}$. Then for
$L$ sufficiently large there exists a continuous map
$$
\Psi : J_\rho^{-L} \to (\g_1)_k * (\g_2)_l
$$
such that the composition
$$
(\g_1)_k * (\g_2)_l \quad \stackrel{\Phi_{\l}}{\longrightarrow}\quad
J_{\rho}^{-L} \quad\stackrel{\Psi}{\longrightarrow}\quad (\g_1)_k * (\g_2)_l
$$
is homotopically equivalent to the identity map on $(\g_1)_k *
(\g_2)_l$ provided that $\l$ is large enough.
\end{pro}

\noindent The rest of this section is devoted to the proof of this
proposition.

\

\noindent By Propositions \ref{p:altern} and \ref{p:proj} we know
that either $\dis \psi_k\left(\frac{\wtilde{h}_1
e^{u_1}}{\int_\Sg\wtilde{h}_1 e^{u_1}dV_g}\right)$ or $\dis
\psi_l\left(\frac{\wtilde{h}_2 e^{u_2}}{\int_\Sg\wtilde{h}_2
e^{u_2}dV_g}\right)$ is well defined (or both), since either $\dis
\dkr\left(\frac{\wtilde{h}_1 e^{u_1}}{\int_\Sg\wtilde{h}_1
e^{u_1}dV_g}, \Sigma_k\right) < \e$ or $\dis
\dkr\left(\frac{\wtilde{h}_2 e^{u_2}}{\int_\Sg\wtilde{h}_2
e^{u_2}dV_g}, \Sigma_l\right) < \e$ (or both).

We then set
$$
  d_1 = \dkr\left(\frac{\wtilde{h}_1 e^{u_1}}{\int_\Sg\wtilde{h}_1 e^{u_1}dV_g}, \Sigma_k\right); \qquad \qquad d_2 = \dkr\left(\frac{\wtilde{h}_2 e^{u_2}}{\int_\Sg\wtilde{h}_2 e^{u_2}dV_g}, \Sigma_l\right),
$$
and consider a function $\wtilde r = \wtilde r(d_1, d_2)$ defined as

\begin{equation} \label{def:r} \wtilde{r}(d_1,d_2)=f\left (\frac{d_1}{d_1+d_2} \right
),
\end{equation}
where $f$ is such that

\begin{equation} \label{def:f} f(z)= \left \{ \begin{array}{ll} 0 & \mbox{ if } z \in[0,1/4],\\
2z-\frac{1}2 & \mbox{ if } z \in (1/4, 3/4),\\ 1 &
\mbox{ if } z \in [3/4,1].\end{array} \right.\end{equation}

%

Consider the global retractions $\Pi_1 : \Sigma \to \gamma_1$ and
$\Pi_2 : \Sigma \to \gamma_2$ given in Lemma \ref{new}, and define:
\begin{equation}
\label{eq:psi}
  \Psi(u_1, u_2) = (1- \wtilde r) (\Pi_1)_* \psi_k \left(\frac{\wtilde{h}_1 e^{u_1}}{\int_\Sg\wtilde{h}_1 e^{u_1}dV_g}\right)+ \wtilde r (\Pi_2)_* \psi_l \left(\frac{\wtilde{h}_2
  e^{u_2}}{\int_\Sg\wtilde{h}_2 e^{u_2}dV_g}\right),
\end{equation}
where $(\Pi_i)_*$ stands for the push-forward of the map $\Pi_i$.
Notice that when one of the two $\psi$'s is not defined the other necessarily
is, and the map is well defined by the equivalence relation.

In what follows, we are going to need the following auxiliary lemma:

\begin{lem} \label{lemmino} Given $n\in \N$, define $\chi_\l$ as
$\dis{\chi_\l(x)= \sum_{i=1}^{n} t_i \left( \frac{\l}{1 + \l^2
d(x,x_i)^2} \right)^2}$. Take a $L^{\infty}$ function $\tau:
\Sigma \to \R$ satisfying:

\begin{enumerate}

\item[i)]  $\tau(x) >m>0$ for all $x \in B(x_i,\delta)$.

\item[ii)] $|\tau(x)| \leq M$ for all $x \in \Sigma$.

\end{enumerate}

Then, there exist constants $c>0$, $C>0$ depending only on
$\Sigma$, $m$, $M$, such that for every $\lambda>0$,

$$ c_0 \min \left \{ 1, \frac{1}{\l} \right \} < \dkr \left (\frac{\tau \, \chi_{\lambda}}{\int_{\Sigma} \tau \, \chi_{\l} \, dV_g}, \Sigma_n \right ) < \frac{C_0}{\l}. $$

\end{lem}

\begin{pf}

We show the proof for $n=1$; the general case uses the same ideas
and will be skipped. We also assume $\l>1$. First of all, observe
that $$C> \displaystyle \int_{\Sigma} \chi_\l(x) \, dV_g(x)
>c>0$$ for some positive constants $c$, $C$.

For the upper estimate, it suffices to show that for any $f$
Lipschitz, $\|f\|_{Lip(\Sigma)} \leq 1$,

$$ \int_{\Sigma} \tau(x)  \left(  \frac{\l}{1 + \l^2
d(x,x_0)^2} \right)^2 (f(x)-f(x_0)) \, dV_g(x) \leq
\frac{C}{\lambda}.
$$

Indeed, by ii),
$$ \int_{(B_\delta(x_0))^c} \tau(x) \left( \frac{\l}{1 + \l^2
d(x,x_0)^2} \right)^2 \, dV_g(x) \leq \frac{C}{\lambda^2},$$ and
using geodesic coordinates $x$ centered at $x_0$, we find
$$\left| \int_{B_\delta(x_0)} \tau(x) \left( \frac{\l}{1 + \l^2
d(x,x_0)^2} \right)^2 (f(x)-f(x_0)) \, dV_g(x) \right| $$$$\le C
\int_{B_{\d \l}(0))} \tau \left( x_0 + \frac{y}{\l} \right )\left(
\frac{1}{1 + |y|^2} \right)^2 \left| f \left(x_0+
\frac{y}{\l}\right )-f(x_0)\right| \, dy
$$$$\leq C\int_{\R^2} \left( \frac{1}{1 + y^2} \right)^2  \left
|\frac{y}{\l}\right | \, dy \leq \frac{C}{\lambda}.$$

We now prove the estimate from below. Given $p \in \Sg$, we
estimate $\dkr(\chi_\l,\ \delta_p)$. Define the Lipschitz function
$f(x)=d(x,p)$. We now show that:

$$ \min_{p\in \Sigma} \int_{\Sigma} \tau(x) \left( \frac{\l}{1 + \l^2
d(x,x_0)^2} \right)^2 d(x,p) \, dV_g(x) \geq \frac{c}{\lambda}.
$$

As above, the integral in the exterior of $B_\d(x_0)$ is
negligible. Moreover, in the same coordinates as above, and taking
into account i), we obtain:

$$ \int_{B_\d(x_0)} \tau(x) \left(\frac{\l}{1 + \l^2
|x-x_0|^2} \right)^2 d(x,p) \, dV_g(x) \sim \int_{B_{ \d \l}(0)}
\tau(x)\left (x_0+\frac{y}{\l} \right) \left( \frac{1}{1 + |y|^2}
\right)^2 \left | x_0-p+\frac{y}{\l} \right | \, dy $$$$\geq
 \frac{m}{\lambda} \int_{B_\delta(0)} \left( \frac{1}{1 + |y|^2} \right)^2 \left |
y+\l(x_0-p)\right | \, dy.$$

It suffices to show that we cannot choose $p_\l$ so that

\begin{equation} \label{exp} \int_{B_\delta(0)}  \left( \frac{1}{1 + |y|^2}
\right)^2 \left | y+\l(x_0-p_\l)\right | \, dx \to 0 \mbox{ as }
\l \to +\infty. \end{equation}

Indeed, if $\l|x_0-p_\l| \to +\infty$, the expression \eqref{exp}
diverges. If not, we can assume that $\l(x_0-p_\l) \to z \in
\R^2$. Then, \eqref{exp} converges to $$\int_{B_\delta(0)}
 \left( \frac{1}{1 + |y|^2} \right)^2 \left |
y+z\right | \, dx >0.$$ which concludes the proof.
\end{pf}

\

\noindent From the previous lemma we deduce the following

\begin{pro} \label{p:dist} Let $\var_i$ be defined by \eqref{e:pl}. Then there exist constants $c>0$, $C>0$
such that for every $\lambda>1$ and every $r \in (0,1)$ one has
$$c_0\min\left\{1,\frac{1}{\l_{1,r}}\right\} \le \dkr \left (\frac{\wtilde{h}_1 e^{\var_1}}{\int_{\Sigma} \wtilde{h}_1 e^{\var_1} \, dV_g}, \Sigma_k \right ) \le \frac{C_0}{\l_{1,r}}; \qquad  c_0\min\left\{1,\frac{1}{\l_{2,r}}\right\} \le \dkr \left (\frac{\wtilde{h}_2 e^{\var_2}}{\int_{\Sigma} \wtilde{h}_2 e^{\var_2} \, dV_g}, \Sigma_l \right ) \le \frac{C_0}{\l_{2,r}}.$$
\end{pro}

\begin{pf}
Clearly, it suffices to prove the estimates for $\dis{\var_1}$ in
the case $\l_{1,r}>1$. By the normalization, it suffices to prove
it to the function $\varsigma=\var_1- 2 \log \left(\l_{1,r} \max
\{1,\ \l_{2,r}\} \right)$.

Observe now that we can write $e^{\varsigma} =
\chi_{\l_{1,r}}(x)\, \tau(x)$, with:

$$ \tau(x) =  \wtilde{h}_1(x) \left [ \sum_{j=1}^{l} s_j \left( \frac{\max
\{1,\ \l_{2,r}\}^2}{1 + \l_{2,r}^2 d(x,y_j)^2} \right)^2
\right]^{-1/2}.$$

It suffices to show that $\tau$ satisfy the conditions of Lemma
\ref{lemmino} to conclude.

\end{pf}

\

\noindent We are now in position to prove that the composition $\Psi\circ\Phi_\l$ is homotopic
to the identity, where $\Psi$ is as in $\eqref{eq:psi}$ and $\Phi_\l(\zeta)=\var_{\l,\zeta}$ is as in $\eqref{e:pl}$. Take $\zeta=(1-r) \s_1 + r \s_2 \in (\g_1)_k
* (\g_2)_l$, with
$$\s_1= \sum_{i=1}^k t_i \d_{x_i},\qquad \ \ \s_2= \sum_{j=1}^l
s_j\d_{y_j}.$$

Set $\displaystyle d_1= \dkr\left(\frac{\wtilde h_1
e^{\varphi_1}}{\int_\Sigma\wtilde h_1 e^{\varphi_1}dV_g},
\Sigma_k\right)$, $\displaystyle d_2=\dkr \left(\frac{\wtilde h_2
e^{\varphi_2}}{\int_\Sigma\wtilde h_2 e^{\varphi_2}dV_g},
\Sigma_l\right)$. By the previous proposition and the definition
of $\l_{1,r}, \l_{2,r}$, there exist constants $0<c_0<C_0$ such
that
$$ c_0 \min \left \{1,\frac{1}{\l (1-r)} \right \} \leq d_1 \leq \frac{C_0}{\l (1-r)}, \ \
c_0 \min \left  \{1,\frac{1}{\l r} \right  \} \leq d_2 \leq
\frac{C_0}{\l r}. $$

Observe then that at least one between $d_1$ and $d_2$ must be smaller than
$\frac{2C_0}{\l}$. Given $\delta>0$ sufficiently small, we have:

$$ r < \delta \Rightarrow \left\{\begin{array}{ll}\frac{d_1}{d_1+d_2} \leq  \displaystyle \frac{\frac{C_0}{\l (1-r)}}{\frac{c_0}{\l (1-r)} + \frac{c_0}{\l
r}}= \frac{C_0}{c_0}r& \mbox{ if } \l \, r \geq 1;\\\frac{d_1}{d_1+d_2} \leq \displaystyle \frac{\frac{C_0}{\l (1-r)}}{c_0 + \frac{c_0}{\l (1-r)}}\leq \frac{C_0}{c_0} \frac{1}{\l}& \mbox{ if } \l \, r \leq
1. \end{array}\right.
$$

In any case, by choosing $\l$, $\delta$ adequately, we obtain that
$\wtilde{r}=0$. This fact is important, since the projection $\dis \psi_l
\left(\frac{\wtilde h_2
e^{\varphi_2}}{\int_\Sigma\wtilde h_2 e^{\varphi_2}dV_g}\right)$ could not be well defined.

Analogously, we have that if $r>(1-\delta)$, then the projection
$\dis \psi_k \left(\frac{\wtilde h_1
e^{\varphi_1}}{\int_\Sigma\wtilde h_1 e^{\varphi_1}dV_g}\right)$ could not be well defined, but
$\wtilde{r}=1$. Moreover, if $\delta \leq r \leq (1-\delta)$, then
$d_i \leq \frac{C}{\delta \l}$, and hence both projections $\dis \psi_k
\left(\frac{\wtilde h_1
e^{\varphi_1}}{\int_\Sigma\wtilde h_1 e^{\varphi_1}dV_g}\right)$, $\dis \psi_l \left(\frac{\wtilde h_2
e^{\varphi_2}}{\int_\Sigma\wtilde h_2 e^{\varphi_2}dV_g}\right)$
are well defined.

\medskip Letting $\wtilde{\zeta}_{\l}= \Psi \circ \Phi_\l (\zeta)=
(1-\wtilde{r}_{\l}) \wtilde{\s}_{1,\l} + \wtilde{r}_\l \wtilde{\s}_{2,
\l}$, we consider  the following homotopy:
$$H_1: (0,1] \times \left ( (\g_1)_k * (\g_2)_l \right )\to \left ((\g_1)_k *
(\g_2)_l \right ), $$$$H_1(\mu, (1-r) \s_1 + r \s_2)=
(1-r_{\mu,\lambda}) \wtilde{\s}_{1,\frac{\l}\mu} + r_{\mu,\lambda}
\wtilde{\s}_{2, \frac{\l}\mu},$$

where $r_{\mu,\lambda}= (1-\mu) f(r) + \mu \wtilde{r}_\lambda$,
and $f$ is given by \eqref{def:f}. Observe that $H_1(1,\cdot) =
\Psi \circ \Phi_{\l}$.

Suppose now $\mu$ tends to zero. Then, as $\l$ is fixed, $\frac{\l}\mu \to
+\infty$, and hence $\displaystyle \frac{\wtilde h_i e^{\varphi_{i,
\frac{\l}\mu}}}{\int_{\Sigma} \wtilde h_i e^{\varphi_{i,\frac{\l}\mu}}dV_g} \rightharpoonup \s_i$.
Proposition \ref{p:projbar} implies that $\displaystyle\psi_k\left(\frac{\wtilde h_1
e^{\varphi_1}}{\int_\Sigma\wtilde h_1 e^{\varphi_1}dV_g}\right) \to \s_1$, $\displaystyle\psi_l\left(\frac{\wtilde h_2 e^{\varphi_2}}{\int_\Sigma\wtilde h_2 e^{\varphi_2}dV_g}\right) \to \s_2$.
Since $\Pi_i$ are retractions, we conclude that $\wtilde{\s}_{i,
\frac \l \mu} \to \s_i$. In other words,

$$ \lim_{\mu \to 0} H_1(\mu, (1-r) \s_1 + r \s_2)= (1-f(r)) \s_1 + f(r) \s_2.$$

We now define:

$$H_2: [0,1] \times \left ( (\g_1)_k * (\g_2)_l \right )\to \left ((\g_1)_k *
(\g_2)_l \right ), $$$$ H_2(\mu, (1-r) \s_1 + r \s_2)= [1-(\mu
f(r) + (1-\mu) r)] \s_1 + (\mu f(r) + (1-\mu) r) \s_2.$$

The concatenation of $H_1$ and $H_2$ gives the desired homotopy.

\section{Min-max scheme}\label{s:minmax}

\noindent We now introduce the variational scheme which yields existence of
solutions: this remaining part follows the ideas of \cite{djlw}
(see also \cite{mal}).

By Proposition \ref{p:en}, given any $L > 0$, there exists $\lambda$ so large that
$J_{\rho}(\var_{\l, \zeta}) <  - L$ for any $\zeta \in (\g_1)_k * (\g_2)_l$. We choose
$L$ so large that Proposition \ref{p:proj} applies: we then have that
the following composition
$$
(\g_1)_k * (\g_2)_l \quad \stackrel{\Phi_\lambda}{\longrightarrow}\quad J_\rho^{-L}
\stackrel{\Psi}{\longrightarrow}\quad (\g_1)_k * (\g_2)_l
$$
is homotopic to the identity map. In this situation it is said
that the set $J_\rho^{-L}$ {\em dominates} $(\g_1)_k * (\g_2)_l$ (see
\cite{hat}, page 528). Since $(\g_1)_k
* (\g_2)_l$ is not contractible, this implies that
$$ \Phi_\lambda((\g_1)_k * (\g_2)_l) \mbox{ is not contractible in }
J_\rho^{-L}.$$ Moreover, we can take $\l$ larger so that $\Phi_\lambda((\g_1)_k
* (\g_2)_l) \subset J_\rho^{-2L}$.

Define the topological cone with basis
$(\g_1)_k * (\g_2)_l$ via the equivalence relation
$$\C = \frac{(\g_1)_k * (\g_2)_l\times[0,1]}{(\g_1)_k * (\g_2)_l\times\{0\}}:$$
notice that, since $(\g_1)_k * (\g_2)_l\simeq \S^{2k+2l-1}$, then $\C$ is homeomorphic to a Euclidean ball of dimension $2k+2l$.

We now define the min-max value:
$$m = \inf_{\xi \in \Gamma} \max_{u \in \C} J(\xi(u)),$$
where \begin{equation} \label{GammaGamma} \Gamma= \{ \xi: \C \to
H^1(\Sigma)\times H^1(\Sigma):\ \xi(\zeta)=\var_{\lambda,\zeta} \ \forall\ \zeta \in
\partial \C\}.\end{equation} Observe that $t\Phi_\lambda: \C \to H^1(\Sigma)\times H^1(\Sigma)$ belongs to $\Gamma$, so this is
a non-empty set. Moreover,

$$ \sup_{\zeta \in \partial \C} J_\rho(\xi(\zeta)) =  \sup_{\zeta \in (\g_1)_k * (\g_2)_l}
J_\rho(\var_{\lambda,\zeta}) \leq -2L.$$

We now show that $m \geq -L$. Indeed, $\partial \C$ is contractible in $\C$, and hence in $\xi(\C)$ for any $\xi \in \Gamma$. Since $\partial \C$ is not contractible in $J_\rho^{-L}$, we conclude that $\xi(\C)$ is not contained in $J_\rho^{-L}$.
Being this valid for any arbitrary $\xi \in \Gamma$, we conclude
that $m \geq -L$.

\medskip

From the above discussion, the functional $J_{\rho}$ satisfies the
geometrical properties required by min-max theory. However, we
cannot directly conclude the existence of a critical point, since
it is not known whether the Palais-Smale condition holds or not.
The conclusion needs a different argument, which has been used
intensively (see for instance \cite{djlw, dm}), so we will be
sketchy.

\

We take $\nu > 0$ such that $$ [\rho_1-2\nu, \rho_1+2\nu]\times [\rho_2-2\nu, \rho_2+2\nu] \subset \R^2\backslash\Lambda,$$
where $\Lambda$ is the set defined as in Definition $\ref{d:glob}$.

Consider now the parameter $\mu \in [-\nu, \nu]$. It is clear that
the min-max scheme described above works uniformly for any $\mu$
in this range. In other words, for any $L > 0$, there exists $\l$
large enough so that
\begin{equation}
   \sup_{\zeta \in \partial
   \C} J_{\wtilde{\rho}}(\xi(\zeta)) < - 2 L; \qquad
   \qquad m_{\mu} := \inf_{\xi \in \Gamma}
  \; \sup_{\zeta \in \C} J_{\wtilde{\rho}}(\xi(\zeta)) \geq -
  L, \ \ \wtilde{\rho}=(\rho_1+\mu, \rho_2+\mu ).
\end{equation}

\noindent In this way, we are led to a problem depending on the
parameter $\mu$ that satisfies a uniform min-max structure. In
this framework, the following Lemma is well-known, usually taking
the name {\em monotonicity trick}. This technique was first used by
Struwe in \cite{struwe88}; a first abstract version was made in
\cite{jeanjean} (see also \cite{djlw, lucia}).

\begin{lem}
There exists $\Upsilon \subset [-\nu,\nu]$ satisfying:
\begin{enumerate}
\item $\left | [-\nu,\nu] \setminus \Upsilon \right | =0$.

\item For any $\mu \in \Upsilon$, the functional
$J_{\wtilde{\rho}}$ possesses a bounded Palais-Smale sequence
$(u_{1,n}, u_{2,n})_n$ at level $m_{\mu}$.

\end{enumerate}
\end{lem}

\

\noindent {\bf Conclusion.} Consider first $\mu \in \Upsilon$.
Passing to a subsequence, the bounded Palais-Smale sequence can be
assumed to converge weakly. Standard arguments show that the weak
limit is indeed strong and that it is a critical point of
$J_{\wtilde{\rho}}$.

Consider now $\mu_n \in \Upsilon$, $\mu_n \to 0$, and let
$(u_{1,n}, u_{2,n})$ denote the corresponding solutions. It is
then sufficient to apply the compactness result in Theorem
\ref{t:lwz}, which yields convergence of $(u_{1,n}, u_{2,n})$ to a
solution of \eqref{eq:todaregul}.

\section{The mean field equation}\label{s:sh}

\noindent In this section we consider the mean field equation, namely the Liouville-type equation
\be
\label{mf-eq}
  - \D u =  \rho_1 \left( \frac{h e^{u}}{\int_\Sg
      h e^{u} dV_g} - 1 \right) - \rho_2 \left( \frac{h e^{-u}}{\int_\Sg
      h e^{-u} dV_g} - 1 \right)
\ee
where $\rho_1, \rho_2$ are two non-negative parameters and $h$ is a smooth positive function. Applying the same analysis developed for the Toda system we give a general existence result.

\

This equation arises in mathematical physics as a mean field equation of the equilibrium turbulence with arbitrarily signed vortices. The mean field limit was first studied by Joyce and Montgomery \cite{jm} and  by Pointin and Lundgren \cite{pl} by means of different statistical arguments. Later, many authors adopted this model, see for example \cite{ch}, \cite{lio}, \cite{new} and the references therein.
The case $\rho_1 = \rho_2$ plays also an important role in the study of constant mean curvature surfaces, see \cite{w1}, \cite{w2}.

Equation \eqref{mf-eq} has a variational structure with associated functional $\wtilde{I}_\rho: H^1(\Sg) \to \R$, with $\rho = (\rho_1, \rho_2)$, defined by
\begin{equation} \label{func}
\wtilde{I}_\rho(u) =  \frac{1}{2}\int_\Sigma |\nabla_g u|^2 \,dV_g - \rho_1 \left( \log\int_\Sigma h \, e^u \,dV_g  - \int_\Sigma u \,dV_g \right)
           -  \rho_2 \left( \log\int_\Sigma h \, e^{-u} \,dV_g + \int_\Sigma u \,dV_g \right).
\end{equation}
In \cite{os1} the authors derived a Moser-Trudinger inequality for $e^u$ and $e^{-u}$ simultaneously, namely
$$
    \log\int_\Sigma e^{u-\ov{u}} \,dV_g  + \log\int_\Sigma e^{-u+\ov{u}} \,dV_g \leq \frac{1}{16\pi}\int_\Sigma |\nabla_g u|^2 \,dV_g + C,
$$
with $C$ depending only of $\Sg$. By this result, solutions to \eqref{mf-eq} can be found immediately as global minima of the functional $\wtilde{I}_\rho$ whenever both $\rho_1$ and $\rho_2$ are less than $8\pi$. For $\rho_i \geq 8\pi$ the existence problem becomes subtler and there are very few results.

The  blow-up analysis of \eqref{mf-eq} was carried out in \cite{os1}, \cite{os2} and
\cite{jwyz}, see in particular Theorem 1.1, Corollary 1.2 and Remark 4.5 in the latter paper.
The following quantization property for a blow-up point $\ov{x}$ and a sequence $(u_n)_n$ of solutions relatively to $(\rho_{1,n}, \rho_{2,n})$ was obtained:
\begin{equation}
    \lim_{r \to 0} \lim_{n \to + \infty} \rho_{1,n} \int_{B_r(\ov{x})} h \, e^{u_n} dV_g \in 8 \pi \N, \qquad
    \lim_{r \to 0} \lim_{n \to + \infty} \rho_{2,n} \int_{B_r(\ov{x})} h \, e^{-u_n} dV_g \in 8 \pi \N. \label{quant}
\end{equation}
As for the Toda system, the case of multiples of $8 \pi$ may indeed occur, see \cite{eswe} and \cite{grpi}.

\

\no Let now define the  set $\wtilde{\L}$ by
$$
    \wtilde{\L} =  (8 \pi \N \times \R) \cup (\R \cup 8 \pi \N) \subseteq \R^2.
$$
Combining \eqref{quant} and the argument in Section 1 of \cite{bm} one finds the following result.

\begin{thm}\label{t:cmpt}
Let $(\rho_1, \rho_2)$ be in a fixed compact set of $\R^2 \setminus \wtilde{\L}$. Then the set of solutions to \eqref{mf-eq} is uniformly bounded in $C^{2,\beta}$ for some $\beta > 0$.
\end{thm}

\

\noindent Before proving our main result we collect here some known existence results. The first one is given in \cite{jwyz} and treats the case $\rho_1\in(8\pi,16\pi)$ and $\rho_2<8\pi$. Via a blow up analysis the authors proved existence of solutions on a smooth, bounded, non simply-connected domain $\Sigma$ in $\mathbb{R}^2$ with homogeneous Dirichlet boundary condition. Later, this result is generalized in \cite{zh} to any compact surface without boundary by using variational methods. The strategy is carried out in the same spirit as in \cite{mreview} and \cite{cheikh} for the Liouville equation \eqref{scalar} and the Toda system \eqref{eq:todaregul}, respectively. The proof relies on some improved Moser-Trudinger inequalities obtained in \cite{cl}. The idea is that, in a certain sense, one can recover the topology of low sub-levels of the functional $\wtilde{I}_\rho$ just from the behaviour of $e^u$. Indeed the condition $\rho_2 < 8\pi$ guarantees that $e^{-u}$ does not affect the variational structure of the problem.

The doubly supercritical regime, namely $\rho_i > 8 \pi$, has to be attacked with a different strategy. The only existence result concerning this case has been proved in \cite{jevnikar} via variational methods where the author adapted the analysis developed to study the Toda system in \cite{mr2} for this framework. The main tool is an improved Moser-Trudinger inequality under suitable conditions on the centre of mass and the scale of concentration of both $e^u$ and $e^{-u}$.

\

We will give here a general existence result.

\begin{thm}\label{t:ex2}
Suppose $\Sg$ is not homeomorphic to $\S^2$ nor $\R\P^2$, and that $\rho_i \notin 8 \pi \N$ for $i= 1,2$. Then \eqref{mf-eq} has a solution.
\end{thm}

The proof is an adaptation of the argument introduced for the Toda system. Roughly speaking the role of the function $u_2$ is played by $-u$.

We start by considering the topological set $(\gamma_1)_k * (\gamma_2)_l$, on which we will base the min-max scheme. We take then two curves $\gamma_1, \gamma_2 \in \Sigma$ with the same properties as before. Let $\zeta \in (\gamma_1)_k * (\gamma_2)_l, \zeta = (1-r) \s_1 + r \s_2$, with
$$
  \s_1 := \sum_{i=1}^k t_i \d_{x_i} \in (\gamma_1)_k \qquad
  \hbox{ and } \qquad \s_2 :=  \sum_{j=1}^l s_j \d_{y_j} \in
  (\gamma_2)_l.
$$
We define now a test function labelled by  $\zeta
\in (\gamma_1)_k * (\gamma_2)_l$, namely for large $L$ we will find a non-trivial map
$$ \widetilde{\Phi}_\l: (\gamma_1)_k * (\gamma_2)_l \to \wtilde{I}_\rho^{-L}.$$
We set $\widetilde{\Phi}_\l(\zeta)= \var_{\l , \zeta}$ given by
$$
    \var_{\l , \zeta} (x) =   \log \, \sum_{i=1}^{k} t_i \left( \frac{1}{1 + \l_{1,r}^2 d(x,x_i)^2} \right)^2
    -  \log  \, \sum_{j=1}^{l} s_j \left( \frac{1}{1 + \l_{2,r}^2 d(x,y_j)^2} \right)^2,
$$
where $\l_{1,r} = (1-r) \l, \l_{2,r} = r \l$.

The following result holds true.

\begin{pro}
Suppose $\rho_1 \in (8 k \pi, 8 (k+1) \pi)$ and $\rho_2 \in (8 l \pi, 8 (l+1) \pi)$.
Then one has
$$
  \wtilde{I}_{\rho}(\var_{\l , \zeta}) \to - \infty \quad \hbox{ as  } \l \to + \infty
  \qquad \quad \hbox{ uniformly in }  \zeta \in (\gamma_1)_k * (\gamma_2)_l.
$$
\end{pro}

\begin{pf}
The proof is developed exactly as in Proposition \ref{p:en}. Here we just sketch the main features.

We define $\tilde{v}_1,\tilde{v}_2: \Sg \rightarrow \mathbb{R}$ by
\begin{eqnarray*}
    \tilde{v}_1(x) = \log \, \sum_{i=1}^{k} t_i \left( \frac{1}{1 + \l_{1,r}^2 d(x,x_i)^2} \right)^2, \\
    \tilde{v}_2(x) = \log \, \sum_{j=1}^{l} s_j \left( \frac{1}{1 + \l_{2,r}^2 d(x,y_j)^2} \right)^2,
\end{eqnarray*}
so that $\var = \tilde{v}_1 - \tilde{v}_2$.

The Dirichlet part of the functional $\wtilde{I}_\rho$ is given by
$$
    \frac 12 \int_{\Sg} |\n \var|^2 \,dV_g  =  \frac 12 \int_\Sg \bigr(|\n \tilde{v}_1|^2 + |\n \tilde{v}_2|^2 - 2\n \tilde{v}_1 \cdot \n \tilde{v}_2\bigr) \,dV_g \leq  \frac 12 \int_\Sg |\n \tilde{v}_1|^2 \,dV_g + \frac 12 \int_\Sg |\n \tilde{v}_2|^2 \,dV_g + C,
$$
where we have used
$$
\left|\int_\Sg \n \tilde{v}_1 \cdot\n \tilde{v}_2 \,dV_g\right| \leq C.
$$

\noindent We first study the cases $r=0$ and $r=1$, starting from $r=0$. The case $r=1$ can be treated in the same way and will be omitted. Observing that $ \n \tilde{v}_2 = 0$ and taking into account the estimates (\ref{gr1}), (\ref{gr2}) on the gradient of $\tilde{v}_1$ we get
$$
    \frac 12 \int_{\Sg} |\n \var|^2 \,dV_g \leq  16 k \pi\bigr(1 + o_\l(1)\bigr) \log \l + C,
$$
where $o_\l(1) \to 0$ as  $\l \to + \infty$.

Reasoning as in Proposition \ref{p:en} we obtain
$$
  \int_\Sg \var \,dV_g = - 4 \bigr(1 + o_\l(1)\bigr) \log \l;  \,\,\, \log \int_\Sg e^{\var} \,dV_g = - 2 \bigr(1 + o_\l(1)\bigr) \log \l; \,\,\, \log \int_\Sg e^{-\var} \,dV_g = 4 \bigr(1 + o_\l(1)\bigr) \log \l.
$$
Therefore we get
$$
  \wtilde{I}_{\rho}(\var_{\l , \zeta}) \leq \bigr( 16 k \pi - 2 \rho_1 + o_\l(1) \bigr)\log \l + C,
$$
where $C$ is independent of $\l$ and $\s_1 ,\s_2$.

We consider now the case $r\in (0,1)$. We can reason as before to estimate the Dirichlet part by
$$
    \frac 12 \int_{\Sg} |\n \var|^2 \,dV_g \leq 16 k \pi\bigr(1 + o_\l(1)\bigr) \log \bigr(\l_{1,r} + \d_{1,r}\bigr) + 16 l \pi\bigr(1 + o_\l(1)\bigr) \log \bigr(\l_{2,r} + \d_{2,r}\bigr) + C,
$$
where $\d_{1,r}>\d>0$ as $r \to 1$ and $\d_{2,r}>\d>0$ as $r \to 0$. Following the argument in Proposition \ref{p:en} we obtain
$$
    \int_\Sg \var \,dV_g = - 4 \bigr(1 + o_\l(1)\bigr) \log \bigr(\l_{1,r} + \d_{1,r}\bigr) + 4 \bigr(1 + o_\l(1)\bigr) \log \bigr(\l_{2,r} + \d_{2,r})\bigr) + O(1),
$$
$$
    \log \int_\Sg e^{\var} \,dV_g = 4 \log \bigr(\l_{2,r} + \d_{2,r}\bigr) - 2 \log \bigr(\l_{1,r} + \d_{1,r}\bigr) + O(1),
$$
$$
    \log \int_\Sg e^{-\var} \,dV_g = 4 \log \bigr(\l_{1,r} + \d_{1,r}\bigr) - 2 \log \bigr(\l_{2,r} + \d_{2,r}\bigr) + O(1).
$$

\noindent Using these estimates we get
$$
  \wtilde{I}_{\rho}(\var_{\l , \zeta}) \leq \bigr( 16 k \pi - 2 \rho_1 + o_\l(1) \bigr)\log \bigr(\l_{1,r} + \d_{1,r}\bigr) + \bigr( 16 l \pi - 2 \rho_2 + o_\l(1) \bigr)\log \bigr(\l_{2,r} + \d_{2,r}\bigr) + O(1).
$$
By assumption we have $\rho_1 > 8k\pi, \rho_2 > 8l\pi$ and exploiting the fact that $\dis{\max_{r\in[0,1]}\{ \l_{1,r}, \l_{2,r} \}} \to +\infty$ as $\l \to \infty$, we deduce the thesis.
\end{pf}

Once we have this result we can proceed exactly as in Section 4. One gets indeed an analogous improved Moser-Trudinger inequality as in Lemma \ref{l:imprc}. We have just to observe that a local Moser-Trudinger inequality still holds in this case, as pointed out in \cite{jevnikar}.

\begin{lem}
Fix $\delta>0$, and let $\O\Subset\wtilde{\O}\subset\Sg$ be such that
$\dis d\left(\O, \partial \wtilde{\O}\right) \geq \delta$. Then, for any $\varepsilon > 0$
there exists a constant $C = C(\varepsilon, \delta)$ such that for all $u
\in H^1(\Sg)$
$$
     \log \int_{\O}  e^{u- \fint_{\wtilde\O}u \,dV_g}\, dV_g + \log \int_{\O} e^{-u + \fint_{\wtilde\O}u \,dV_g}\, dV_g \leq \frac{1+\e}{16\pi} \int_{\wtilde{\O}} |\nabla_g u|^2 \,dV_g + C.
$$
\end{lem}

Therefore, considering $\rho_1 \in (8 k \pi, 8 (k+1) \pi)$ and $\rho_2 \in (8 l \pi, 8 (l+1) \pi)$, we deduce that on low sub-levels of the functional $\wtilde{I}_\rho$ at least one of the component of $\dis\left(\frac{he^u}{\int_\Sg he^udV_g},\frac{he^{-u}}{\int_\Sg he^{-u}dV_g}\right)$ has to be very close to the sets of $k$- or $l$- barycenters over $\Sg$, respectively, see Proposition \ref{p:altern} for details. It is then possible to construct a continuous map
$$
\wtilde{\Psi} : \wtilde{I}_\rho^{-L} \to (\g_1)_k * (\g_2)_l
$$
with $L$ sufficiently large, such that the composition
$$
(\g_1)_k * (\g_2)_l \quad \stackrel{\wtilde{\Phi}_{\l}}{\longrightarrow}\quad
\wtilde{I}_{\rho}^{-L} \quad\stackrel{\wtilde{\Psi}}{\longrightarrow}\quad (\g_1)_k * (\g_2)_l
$$
is homotopically equivalent to the identity map on $(\g_1)_k *
(\g_2)_l$ provided that $\l$ is large enough. $\wtilde{\Psi}$ is defined as in \eqref{eq:psi}, where basically $e^{u_2}$ is replaced by $e^{-u}$:
$$
  \wtilde{\Psi}(u) = (1 - \wtilde r) (\Pi_1)_* \psi_k \left(\frac{he^{u}}{\int_\Sg he^{u}dV_g}\right)+  \wtilde r  (\Pi_2)_* \psi_l \left(\frac{he^{-u}}{\int_\Sg he^{-u}dV_g}\right).
$$
With this at hand we argue as in Section 5 introducing a min-max scheme based on the set $(\g_1)_k *
(\g_2)_l$. Allowing $(\rho_1, \rho_2)$ to vary in a compact set of $(8 k \pi, 8 (k+1) \pi) \times (8 l \pi, 8 (l+1) \pi)$ we obtain a sequence of solutions $(u_n)_n$ corresponding to $(\rho_{1,n}, \rho_{2,n}) \to (\rho_1, \rho_2)$. We finally get a solution for $(\rho_1, \rho_2)$ by applying the compactness result in Theorem \ref{t:cmpt}.

\section{Appendix: On the CW structure of Barycenter Spaces, by Sadok Kallel}

\noindent In this appendix we show that barycenter spaces of CW-complexes
are again CW. The notation of this appendix is independent of the rest of the paper, and the proofs use arguments from algebraic topology.

\

We adopt the notation $\bar{n}$ for barycenter and
$\sj{n}$ for symmetric join, see \cite{kk}. We also need the notation
$$\Delta_{k-1} = \{(t_1,\ldots, t_k)\ \in [0,1]^k\ |\ \sum
t_i=1\}$$ for the $(k-1)$-dimensional complex. This we view as a
CW-complex with faces being subcomplexes. For $k<n$, we write as
$\Delta_{k-1}\hookrightarrow\Delta_{n-1}$ the standard face
inclusion given by adjoining trivial coordinate entries
$(t_1,\ldots, t_k)\mapsto (t_1,\ldots, t_k,0,\ldots, 0)$.
Similarly for based $X$, with basepoint $x_0$, we embed
$X^k\hookrightarrow X^n$ by adjoining basepoints.

\bpr\label{main} If $X$ is a based connected CW-complex, then
$\bar{n}(X)$ can be equipped with a CW structure so that all
vertical projections in the following diagram are cellular maps
and all horizontal maps are subcomplex inclusions
$$\xymatrix{\Delta_{k-1}\times X^k\ar@{^(->}[r]\ar[d]&\Delta_{n-1}\times X^n\ar[d]\\
\bar{k}(X)\ar@{^(->}[r]&\bar{n}(X). }$$ \epr

The proof uses standard facts about CW complexes which we now
review.

\begin{enumerate}
\item If $(X,A)$ is a relative CW complex, then the quotient space
$X/A$ is a CW complex with a vertex corresponding to $A$. \item
More generally if $A$ is a subcomplex of a CW complex $X, Y$ is a
CW complex, and $f: A\longrightarrow  Y$ is a cellular map, then
the {\em pushout} $Y\cup_fX$ has an induced CW complex structure that
contains $Y$ as a subcomplex and has one cell for each cell of $X$
that is not in $A$. We represent this construction by a diagram
    $$\xymatrix{A\ar@{^(->}[r]^i\ar[d]^f&X\ar[d]\\
    Y\ar[r]&X\cup_fY}$$
    with the understanding that all maps arriving at $X\cup_fY$ are cellular with respect to the induced cell structure there.
\item A finite group, or more generally a discrete group $G$ acts
{\em cellularly} on X means that: (i) if $\sigma$ is an open cell of
$X$ then $g\sigma$ is again an open cell in $X$ for all $g\in G$,
and (ii) if $g\in G$ fixes an open cell $\sigma$, that is $g\sigma
= \sigma$, then it fixes $\sigma$ pointwise (i.e. $gx=x$ for all
$x\in\sigma$). A CW-complex is a {\em cellular $G$-space} if $G$ acts
cellularly on $X$. If a finite group $G$ acts cellulary on $X$,
then $X/G$ is a CW-complex. Furthermore, if $f: X\longrightarrow
Y$ is a $G$-equivariant cellular map between cellular $G$-spaces,
then the induced map $X/G\longrightarrow  Y/G$ is cellular with
respect to the induced CW-structures.
\end{enumerate}

Properties (1) and (2) can be found in (\cite{may}, Chapter 10.2).
Property (3) follows from Prop. 1.15 and Ex. 1.17 of \cite{td}
(Chapter 2). Throughout we endow $X$ with a CW-structure so that
the permutation action of $\sn$ on $X^n$ is cellular,
and so that $x_0$ is a $0$-cell or vertex.\\

\noindent{\sl Proof of Proposition \ref{main}}. We recall the
definition of the barycenter spaces. Given $X$ a space, then its
$n$-th barycenter space is the quotient space
$$\bar{n}(X) := \coprod_{k=1}^n\Delta_{k-1}\times_{\mathfrak S_k} X^k/_{\sim}$$
where $\Delta_{k-1}\times_{\mathfrak S_k} X^k$ is the quotient of
$\Delta_{k-1}\times X^k$ by the symmetric group $\mathfrak S_k$
acting diagonally, and where $\sim$ is the equivalence relation
generated by:
\begin{eqnarray*}(i)
&&[t_1,\ldots, t_{i-1},0,t_{i+1},\ldots, t_n;x_1,\ldots,  x_i,\ldots x_n]\\
&&\sim [t_1,\ldots, t_{i-1},t_{i+1},\ldots, t_n;x_1,\ldots,\hat
x_i,\ldots x_n]
\end{eqnarray*}
(here $\hat x_i$ means the $i$-th entry has been suppressed), and
by
\begin{eqnarray*}(ii)
&&[t_1,\ldots, t_i,\ldots, t_j,\ldots, t_n;x_1,\ldots,x_i,\ldots , x_j,\ldots, x_n]\\
&&\sim [t_1,\ldots, t_{i-1}, t_i+t_j,t_{i+1},\ldots, \hat
t_j,\ldots, t_n; x_1,\ldots, x_i,\ldots, \hat x_j,\ldots ,x_n]\ \
\ \ \ \hbox{if}\ \ \ x_i=x_j.
\end{eqnarray*}
An intermediate construction is to consider the symmetric join
$\sj{n}(X)$ which is the quotient of
$\coprod_{k=1}^n\Delta_{k-1}\times_{\mathfrak S_k} X^k$ by the
equivalence relation $(i)$ only. There are quotient projections
$$\Delta_{n-1}\times X^n\longrightarrow \Delta_{n-1}\times_\sn X^n\longrightarrow \sj{n}(X)\longrightarrow \bar{n}(X)$$
and it is convenient to write an equivalence class in
$\Delta_{n-1}\times_\sn X^n$ or any of its images in $\sj{n}X$ and
$\bar{n}(X)$ by
$$\sum_{i=1}^n t_ix_i := [t_1,\ldots ,t_n;x_1,\ldots, x_n].$$
{\em Addition} means the sum is abelian and this reflects the
symmetric group action. The relation $(i)$ means the entry $0x_i$ is
suppressed, and relation $(ii)$ means that $t_ix+t_jx = (t_i+t_j)x$.

To show that $\bar{n}(X)$ is CW, we proceed by induction. When
$n=1$, $\bar{1}X=X$ so there is nothing to prove. For the general
case, write
$$\bar{n}X = \bar{n-1}X\cup \left(\Delta_{n-1}\times_\sn X^n\right)/_\sim$$
and write $X^n_{fat}\subset X^n$ the fat diagonal consisting of
all n-tuples $(x_1,\ldots, x_n)$ with $x_i=x_j$ for some $i\neq j$.
Denote by
$$W_n=\left(\partial\Delta_{n-1}\times_\sn X^n\right)\ \bigcup\ \left(\Delta_{n-1}\times_\sn X^n_{fat}\right)$$
the subspace of $\Delta_{n-1}\times_\sn X^n$ consisting of all
classes $\sum t_ix_i$ with $t_i=0$ for some $i$ or $x_i=x_j$ for
some $i\neq j$. Then $W_n$ is a CW subcomplex of $X^n$ because the
$\sn$-equivariant decomposition of $X^n$ can always be arranged so
that $\Delta_{fat}$ is a subcomplex. There is a well-defined
quotient map $f: W_n\longrightarrow \bar{n-1}$ sending
\begin{eqnarray*}\label{first}
\sum t_jx_j&\longmapsto& \sum_{j\neq i} t_jx_j\ \ \ \hbox{if}\ t_i=0\\
\sum t_jx_j&\longmapsto& t_1x_1+\cdots + (t_i+t_j)x_i + \cdots +
\widehat{t_jx_j} + \cdots + t_nx_n\ \ \hbox{if}\ x_i=x_j
\end{eqnarray*}
and we have the pushout diagram
\begin{equation*}(*)
\xymatrix{
W_n\ar@{^(->}[r]\ar[d]^f&\Delta_{n-1}\times_\sn X^n\ar[d]\\
\bar{n-1}X\ar[r]&\bar{n}(X). }\end{equation*} If we can show that
$f$ is cellular, then by property (2) and induction, $\bar{n}(X)$
will be CW as desired.

The map $f$ has two restrictions $f_1$ and $f_2$ on the pieces
$\partial\Delta_{n-1}\times_\sn X^n$ and $\Delta_{n-1}\times_\sn
X^n_{fat}\subset W_n$ respectively. To see that $f_1$ is cellular,
write $\partial\Delta_{n-1}$ as a union of faces
$F_i=\{(t_1,\ldots, t_n), t_i=0\}$ each homeomorphic to
$\Delta_{n-2}$. Write $X^n_i=\{(x_1,\ldots, x_n)\in X^n\ |\
x_i=x_0\}$ where $x_0\in X$ is the basepoint. The maps $F_i\times
X^n\longrightarrow  F_i\times X^n_i$, $(t_1,\ldots,
t_n;x_1,\ldots, x_n)\mapsto (t_1,\ldots, t_n;x_1,\ldots,
x_0,\ldots ,x_n)$; which for a given $i$ replaces $x_i$ by $x_0$,
are cellular and so is their union
$$\bigcup_i F_i\times X^n\longrightarrow  \bigcup_i F_i\times X^n_i.$$
This map is $\sn$-equivariant and so passes to a cellular map
between quotients
$$\xymatrix{\left(\bigcup_iF_i\times X^n\right)/_{\sn}\ar[r]\ar@{=}[d]&\left(\bigcup_i F_i\times X^n_i\right)/_{\sn}\ar@{=}[d]\\
\partial\Delta_{n-1}\times_\sn X^n\ar[r]^g&\Delta_{n-2}\times_{\mathfrak S_{n-1}}X^{n-1}.
}$$ The restriction $f_1$ is now the composite of cellular maps
$$\xymatrix{\partial\Delta_{n-1}\times_\sn X^n\ar[r]^g& \Delta_{n-2}\times_{\mathfrak S_{n-1}}X^{n-1}\ar[r]& \bar{n-1}(X)}$$
thus it is cellular. We proceed the same way for the restriction
$f_2$. Write $X^n_{fat} = \bigcup_{i<j} X^n_{ij}$ where $X^n_{ij}
= \{(x_1,\ldots, x_n)\in X^n\ |\ x_i=x_j, i<j\}$. Each $X^n_{ij}$
is identified with $X^{n-1}$. There are maps $\tau_{ij}:
\Delta_{n-1}\times X^n_{ij}\longrightarrow  F_i\times X^n_i$
sending
\begin{eqnarray*}
&&(t_1,\ldots, t_n,x_1,\ldots, x_n)\\
&&\longmapsto (t_1,\ldots, t_{i-1},0,t_{i+1},\ldots, t_{j-1},
t_i+t_j, t_{j+1},\ldots, t_n;x_1,\ldots,
x_{i-1},x_0,x_{i+1},\ldots ,x_n)
\end{eqnarray*}
which are cellular being the product of cellular maps (i.e it can
be checked that the map $\Delta_{n-1}\longrightarrow
\partial\Delta_{n-1}$ sending $(t_1,\ldots, t_n)\longrightarrow
(t_1,\ldots, t_{i-1},0,t_{i+1},\ldots, t_{j-1}, t_i+t_j,
t_{j+1},\ldots, t_n)$ sends faces to faces and hence is cellular).
The map $\bigcup\tau_{ij}$ is not $\sn$-equivariant, but the
composite
$$\xymatrix{\bigcup_{i<j}\Delta_{n-1}\times X^n_{ij}\ar[r]&\bigcup_i F_i\times X^n_i\ar[r]&\left(\bigcup_i F_i\times X^n_i\right)/_{\sn}}$$
factors through the $\sn$-quotient. More precisely, we have the
diagram
$$\xymatrix{\left(\bigcup_{i<j}\Delta_{n-1}\times X^n_{ij}\right)/_{\sn}\ar[r]\ar@{=}[d]&\left(\bigcup_i F_i\times X^n_i\right)/_{\sn}\ar@{=}[d]\\
\Delta_{n-1}\times_\sn
X^n_{fat}\ar[r]^\tau&\Delta_{n-2}\times_{\mathfrak S_{n-1}}X^{n-1}
\ar[r]&\bar{n-1}(X)}$$ with all maps in this diagram cellular. The
bottom composite $f_2$ must therefore be cellular.

 In conclusion, the map $f = f_1\cup f_2$ in the diagram (*) is cellular and this completes the proof.
\hfill$\Box$

\bex We take a special look at $\bar{2}(X)$. Consider $\sj{2}X$
which consists of elements of the form $t_1x + t_2y$ with
$t_1+t_2=1$ and the identification $0x+1y =y$. By using the order
on the $t_i$'s in $I=[0,1]$, this can be written as
\begin{eqnarray*}
\sj{2}(X) &=& \{(t_1,t_2,x_1,x_2)\ | \ t_1\leq t_2, t_1+t_2=1\}/_\sim\\
&=& J\times (X\times X)/_\sim
\end{eqnarray*}
where $J=\{0\leq t_1\leq t_2\leq 1, t_1+t_2=1\}$ is a copy of the
one-simplex, and the identification $\sim$ is such that
$(0,1,x,y)\sim (0,1,x',y)$ and $({1\over 2},{1\over 2}, x,y)\sim
({1\over 2},{1\over 2}, y,x)$. Note that $(0,1)$ and $({1\over
2},{1\over 2})$ are precisely the faces or endpoints of $J$. This
is saying that $\sj{2}X$ is precisely the double mapping cylinder
$$\xymatrix{X^2\times \{(0,1)\}\sqcup X^2\times \{({1\over 2},{1\over 2})\}\ar@{^(->}[r]\ar[d]^{p_2\sqcup\pi}&X^2\times J\ar[d]\\
X\sqcup \hbox{SP}^{2}X\ar[r]&\sj{2}X }$$ where $p_2$ is the
projection onto the second factor $X^2\longrightarrow  X$, and
$\pi$ is the ${\mathbb Z}_2$-quotient map $X^2\longrightarrow
\hbox{SP}^{2}X$ (see \cite{kk}). Both maps $p_2$ and $\pi$ are
cellular (Property (3)). This gives $\sj{2}(X)$ a CW-structure
according to property (3).  We can now consider the pushout
diagram
\begin{equation}\label{bar2}
\xymatrix{J\times X\ar[r]\ar[d]&\sj{2}X\ar[d]\\
X\ar[r]&\bar{2}X}
\end{equation}
where the left vertical map $J\times X\longrightarrow  X$ is
projection hence cellular, while the top map $J\times
X\longrightarrow \sj{2}X$, $((t_1,t_2),x)\mapsto t_1x+t_2x$, is a
subcomplex inclusion. By property (2), $\bar{2}(X)$ is CW. \eex

\bibliography{BJMR3}
\bibliographystyle{plain}

\end{document}